\documentclass[onefignum,onetabnum]{siamart190516}

\usepackage{graphicx}
\usepackage{multirow,bigstrut}
\usepackage{amsmath,amsfonts,amssymb}
\usepackage{mathrsfs}
\usepackage{color}
\usepackage{upgreek}
\usepackage{bm}
\usepackage{verbatim}
\usepackage{mathtools}
\usepackage{calligra}
\usepackage[T1]{fontenc}
\usepackage{subfigure}
\usepackage{pgfplots}
\usepackage{xcolor}

\usepackage{tikz}
  \usetikzlibrary{calc,shapes,decorations,decorations.text, mindmap,shadings,patterns,matrix,arrows,intersections,automata,backgrounds}
  
\usetikzlibrary{shapes,arrows}

\usepackage{verbatim}
\usepackage{exscale}
\usepackage{relsize}
\theoremstyle{plain}
\newtheorem{rem}[theorem]{\hspace{1mm}Remark}

\newcommand{\cm}[1]{{\color{black}{#1}}}

\makeatletter
\tikzset{% customization of pattern 
        hatch distance/.store in=\hatchdistance,
        hatch distance=5pt,
        hatch thickness/.store in=\hatchthickness,
        hatch thickness=5pt
        }
\pgfdeclarepatternformonly[\hatchdistance,\hatchthickness]{north east hatch}% name
    {\pgfqpoint{-1pt}{-1pt}}% below left
    {\pgfqpoint{\hatchdistance}{\hatchdistance}}% above right
    {\pgfpoint{\hatchdistance-1pt}{\hatchdistance-1pt}}%
    {
        \pgfsetcolor{\tikz@pattern@color}
        \pgfsetlinewidth{\hatchthickness}
        \pgfpathmoveto{\pgfqpoint{0pt}{0pt}}
        \pgfpathlineto{\pgfqpoint{\hatchdistance}{\hatchdistance}}
        \pgfusepath{stroke}
    }
\makeatother

\newcommand{\tmop}[1]{\ensuremath{\operatorname{#1}}}

\begin{document}
 
% --------------------------------------------------------------
%                         Start here
% --------------------------------------------------------------

\headers{GFEMs with locally optimal spectral approximations}{C. P. Ma, R. Scheichl, and T. J. Dodwell} 
 
\title{Novel design and analysis of generalized FE methods based on locally optimal spectral approximations}

\author{Chupeng Ma\thanks 
 {Institute for Applied Mathematics and Interdisciplinary Center for Scientific Computing, Heidelberg University, Im Neuenheimer Feld 205, Heidelberg 69120, Germany (\email{chupeng.ma@uni-heidelberg.de}, \email{r.scheichl@uni-heidelberg.de}).}
 \and R. Scheichl\footnotemark[1]
\and T. J. Dodwell\thanks{College of Engineering, Mathematics and Physical Sciences, University of Exeter, Exeter EX4 4PY, United Kingdom, and The Alan Turing Institute, London, NW1 2DB, United Kingdom (\email{T.Dodwell@exeter.ac.uk}).}}

\maketitle

\begin{abstract}
In this paper, the generalized finite element method (GFEM) for solving second order elliptic equations with rough coefficients is studied. New optimal local approximation spaces for GFEMs based on local eigenvalue problems involving a partition of unity are presented. These new spaces have advantages over those proposed in [I. Babuska and R. Lipton, Multiscale Model.\;\,Simul., 9 (2011), pp.~373--406]. First, in addition to a nearly exponential decay rate of the local approximation errors with respect to the dimensions of the local spaces, the rate of convergence with respect to the size of the oversampling region is also established. Second, the theoretical results hold for problems with mixed boundary conditions defined on general Lipschitz domains. Finally, an efficient and easy-to-implement technique for generating the discrete $A$-harmonic spaces is proposed which relies on solving an eigenvalue problem associated with the Dirichlet-to-Neumann operator, leading to a substantial reduction in computational cost. Numerical experiments are presented to support the theoretical analysis and to confirm the effectiveness of the new method.
\end{abstract}



\section{Introduction}\label{sec-1}
Numerous problems in science and engineering involve multiple scales. One example is the flow and transport of fluid within porous media, which often exhibit highly heterogeneous, multiscale variations in both permeability and porosity. Another example is the modelling of composite materials, widely used in high value engineering products, whereby a highly stiff material (e.g.\;\,carbon / graphine) is embedded within a compliant matrix (e.g.\;\,resin). Mathematical modelling of such materials or engineering systems leads to partial differential equations (PDEs) with highly osciallatory coefficients. Whilst, in many cases, only the macroscopic properties of the solution are of interest, they are often strongly influenced by the micro- and mesoscopic details of the media, making a direct discretization at the macroscale unreliable. However, direct numerical solution on a fine mesh that resolves all the small-scale features is computationally expensive and notoriously ill-conditioned \cite{butler2020high}. This motivates the development of {\it multiscale methods} which reduce the computational cost by efficiently incorporating physically important fine-scale information into a coarse-scale representation.

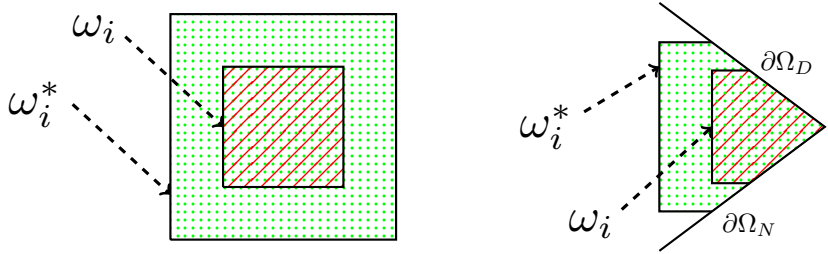
\begin{figure}\label{fig:1-1}
\centering
\begin{tikzpicture}
\draw[pattern=dots, pattern color=green] (0, 0) rectangle (3, 3);
%\draw[fill=white, pattern color=red, pattern=north east lines] (1, 1) rectangle (3, 3); 
\draw[pattern=north east hatch, hatch distance=2.5mm, hatch thickness=.5pt, pattern color=red] (0.7, 0.7) rectangle (2.3, 2.3); 
\draw[->, dashed, very thick] (-0.6,2.7)--(0.7,1.5);
\draw[thick](0.7,0.7)--(2.3,0.7)--(2.3,2.3)--(0.7,2.3)--(0.7,0.7);
\draw[thick](0,0)--(3,0)--(3,3)--(0,3)--(0,0);
\draw[->, dashed, very thick] (-1.3,1.8)--(0,0.6);
\node[scale=1.8] at (-1.0, 2.8) {$\omega_{i}$};
\node[scale=1.8] at (-1.8, 1.8) {$\omega^{\ast}_{i}$};

 \draw[pattern=north east hatch, hatch distance=2.5mm, hatch thickness=.5pt, pattern color=red] (7.7, 2.25)--(7.2, 2.25)--(7.2,0.75)--(7.7,0.75)--(8.7,1.5)--(7.7,2.25); 
 \draw[pattern=dots, pattern color=green] (7.2, 2.625)--(6.5, 2.625)--(6.5, 0.375)--(7.2,0.375)--(8.7,1.5)--(7.2, 2.625);
\draw[thick] (6.5,-0.15)--(8.7,1.5)--(6.5,3.15);
\draw[thick] (7.7, 2.25)--(7.2, 2.25)--(7.2,0.75)--(7.7,0.75);
\draw[thick] (7.2, 2.625)--(6.5, 2.625)--(6.5, 0.375)--(7.2,0.375);

\draw[->, dashed, very thick] (6.0,0.4)--(7.2,1.5);
\draw[->, dashed, very thick] (5.5,1.7)--(6.5,2.3);
   
\node[scale=1.8] at (5.6,0.2) {$\omega_{i}$};
\node[scale=1.8] at (5.0,1.5) {$\omega^{\ast}_{i}$};
\node at (8.2,2.35) {$\partial \Omega_{D}$};
\node at (7.7,0.2) {$\partial \Omega_{N}$};
\end{tikzpicture}
\caption{Illustration of a subdomain that lies within the interior of $\Omega$ (left) or intersects the boundary of $\Omega$ (right) and the associated oversampling domain.}
\end{figure}

The study of multiscale methods has been an active field over the past few decades and various methods have been developed. We restrict our attention here to one class of multiscale methods that aims at constructing localized multiscale basis functions as trial spaces for the finite element method (FEM). In the multiscale finite element method (MsFEM) \cite{hou1997multiscale,chen2003mixed,efendiev2009multiscale}, multiscale basis functions are constructed by solving boundary value problems associated with the original PDE on each coarse-grid block. Convergence of the MsFEM in the periodic setting was proved and an oversampling technique to reduce the resonance error was investigated in \cite{hou1999convergence,henning2013oversampling}. The MsFEM was later generalized to the Generalized Multiscale FEM in \cite{efendiev2013generalized,efendiev2014generalized,chung2015mixed}, where coarse trial spaces were constructed by a spectral decomposition of some snapshot spaces. Another method that has become popular in recent years is the localized orthogonal decomposition (LOD) method \cite{maalqvist2014localization,henning2014localized}. In this method, each nodal basis function of the coarse finite element space is modified with a correction containing fine-scale information. These corrections are first defined as solutions of some global problems, and then proved to decay exponentially fast, which justifies to localize the construction of the correctors. For more studies on alternative multiscale methods, we refer to \cite{hou2016optimal,owhadi2014polyharmonic,
owhadi2011localized,arbogast2006subgrid,ming2005analysis,weinan2007heterogeneous}.

In this paper, we deal with yet another, related multiscale method, the Multiscale Spectral Generalized Finite Element Method (MS-GFEM) introduced in \cite{babuska2011optimal} and further developed in \cite{babuvska2014machine,babuvska2020multiscale}. It is a generalized finite element method (GFEM) with local approximation spaces constructed by solving local spectral problems. The GFEM \cite{babuvska1997partition,melenk1995generalized} proposed by Babuska and Melenk is an extension of the FEM based on a domain decomposition technique combined with a partition of unity approach. In this method, the computational domain is partitioned into a collection of overlapping subdomains $\omega_{i}$ ($i=1,2,\cdots,m$) where the local approximation spaces are built. These local approximation spaces are then "glued together" by a partition of unity to build the trial space for the FEM. One advantage of the GFEM over the FEM is that one can exploit the structure of the PDE under consideration to construct local spaces with much better approximation properties than simple polynomials. In addition, the local computations can be performed in parallel naturally. In \cite{babuska2011optimal}, the solution to be approximated in a subdomain $\omega_{i}$ was decomposed into two orthogonal parts, one part being the solution of a local boundary value problem and the other part belonging to the $A$-harmonic space on $\omega_{i}$, that is
\begin{equation}
H_{A}(\omega_{i}) = \big\{u\in H^{1}(\omega_{i}):\;\int_{\omega_{i}}A({\bm x})\nabla u\cdot\nabla v\,d{\bm x} =0\;\;\forall v\in H_{0}^{1}(\omega_{i})\big\}.   
\end{equation}
Here $A({\bm x})$ is the given $L^{\infty}$-coefficient of the elliptic PDE under consideration. An optimal approximation space for approximating the $A$-harmonic part was constructed by using the characterization of the Kolmogrov $n$-width of a compact restriction operator $P$ from $H_{A}(\omega^{\ast}_{i})$ into $H_{A}(\omega_{i})$. Here, the $\omega_{i}^{\ast}\supset \omega_{i}$ are referred to as the oversampling domains as illustrated in \cref{fig:1-1}. It was shown that the $n$-dimensional optimal approximation space is spanned by the first $n$ eigenfunctions of an eigenvalue problem involving the restriction operator posed in the $A$-harmonic space $H_{A}(\omega^{\ast}_{i})$ and that the approximation converges nearly exponentially with respect to $n$. However, a theoretical investigation of how the local error varies with the size of the oversampling region was missing. Moreover, due to the use of a particular extension technique for boundary subdomains in the proof, the theoretical results in \cite{babuska2011optimal} only hold for problems with pure Dirichlet or Neumann boundary conditions defined on a $C^{1}$-smooth domain. 

Strategies for the numerical implementation of the MS-GFEM were discussed in \cite{babuvska2020multiscale}. The most expensive part of the whole computational work lies in the generation of the discrete $A$-harmonic spaces based on finite element approximations of the spaces $H_{A}(\omega^{\ast}_{i})$ over which the eigenvalue problems are solved. Indeed, the discrete $A$-harmonic space on a domain resolved by a finite element mesh is spanned by the $A$-harmonic extensions of the hat functions corresponding to the $M$ boundary nodes. In \cite{babuvska2020multiscale}, it was suggested that instead of generating an $M$-dimensional discrete $A$-harmonic space, the span of the $A$-harmonic extensions of $N\ll M$ boundary hat functions with wider support can be used as an approximation, which results in fewer local boundary value problems. However, using this method still requires to solve a large number of local problems, especially when the underlying FE mesh is very fine. Furthermore, how to choose the boundary hat functions and their support is a subtle issue in practical implementations.

In this paper, the results of \cite{babuska2011optimal,babuvska2020multiscale} are extended in several respects. First, optimal local approximation spaces for the GFEM based on eigenfunctions of local spectral problems involving a partition of unity are constructed. A similar eigenvalue problem was used to construct coarse spaces for the two-level overlapping Schwarz method with application to PDEs with rough coefficients in \cite{spillane2014abstract}. Instead of introducing a restriction operator as in \cite{babuska2011optimal}, it is shown that the multiplication of a function by one of the partition of unity functions constitutes a compact operator in $H_{A}(\omega^{\ast}_{i})$. In contrast to the traditional GFEM, which approximates the exact solution $u$ in each subdomain, our approach naturally leads to the approximation of $\chi_{i} u$ in each subdomain $\omega_{i}$, where $\chi_{i}$ is the partition of unity function supported on $\omega_{i}$. This makes the estimate of the global approximation error much simpler. Another new feature of our method is that it converges even without oversampling; see Remark \ref{rem:3-4}. Secondly, a sharper error bound for the optimal local approximation is derived. In addition to a nearly exponential decay rate with the dimension of the local spaces, the rate of convergence with respect to the size of the oversampling region is also established. In particular, it is shown that the convergence rate with respect to the dimension of the local spaces becomes higher with increasing oversampling size. Furthermore, the results in this paper hold for problems with mixed boundary conditions defined on general Lipschitz domains. The key to our proof for subdomains near the outer boundary lies in a different definition of the $A$-harmonic spaces on these subdomains, the use of a Caccioppoli-type argument, and a refined analysis of the resulting approximation spaces. Finally, an efficient and easy-to-implement method to generate the discrete $A$-harmonic spaces by solving a Steklov eigenvalue problem on each subdomain is proposed, similar to the one proposed and analysed in the context of the overlapping Schwarz method in \cite{dolean2012analysis}. In particular, the eigenfunctions corresponding to the finite eigenvalues of the Steklov eigenvalue problem span the discrete $A$-harmonic space. Moreover, without using all the eigenfunctions, a small number of discrete $A$-harmonic basis functions provide good numerical results in practice. In this way, the discrete $A$-harmonic spaces can be constructed by solving an eigenvalue problem once at a much lower computational cost than solving many local boundary value problems.

The rest of this paper is organized as follows. In \cref{sec-2}, we describe the problem considered in this paper and give a brief introduction of the GFEM. \Cref{sec-3} is devoted to the construction of the local particular functions and the optimal local approximation spaces. Upper bounds for the local approximation errors are also derived in this section. We discuss the numerical implementation of the method with focus on the construction of the discrete $A$-harmonic spaces in \cref{sec-4}. Numerical examples are given in \cref{sec-5} to validate our theoretical results and the effectiveness of our method.
\section{The GFEM}\label{sec-2}
We consider elliptic PDEs with mixed boundary conditions:
\begin{equation}\label{eq:1-1}
\left\{
\begin{array}{lll}
{\displaystyle -{\rm div}(A({\bm x})\nabla u({\bm x})) = f({\bm x}),\quad {\rm in}\;\, \Omega }\\[2mm]
{\displaystyle {\bm n} \cdot A({\bm x})\nabla u({\bm x})=g({\bm x}), \quad \quad \;\;{\rm on}\;\,\partial \Omega_{N}}\\[2mm]
{\displaystyle u({\bm x}) = q({\bm x}), \quad \qquad \qquad \qquad {\rm on}\;\,\partial \Omega_{D},}
\end{array}
\right.
\end{equation}
where $\Omega\subset \mathbb{R}^{d}$ ($d=2,3$) is a bounded domain with Lipschitz boundary $\partial \Omega$, $\partial\Omega_{D}\cap \partial\Omega_{N} = \emptyset$, and $\overline{\partial\Omega_{D}}\cup \overline{\partial\Omega_{N}} = \partial \Omega$. The vector ${\bm n}$ denotes the unit outward normal. We assume that the matrix $A({\bm x}) = (a_{ij}({\bm x}))_{1\leq i,j\leq d}\in (L^{\infty}(\Omega))^{d\times d}$ is symmetric and there exists $0<\alpha < \beta< +\infty$ such that
\begin{equation}\label{eq:1-1-0}
\alpha \sum_{i=1}^{d} \xi^{2}_{i} \leq \sum_{i,j=1}^{d}a_{ij}({\bm x}) \xi_{i}\xi_{j} \leq \beta  \sum_{i=1}^{d} \xi^{2}_{i},\quad \forall \xi= (\xi_{1},\cdots,\xi_{d})\in \mathbb{R}^{d},\quad {\bm x} \in\Omega.
\end{equation}
We suppose that $f\in L^{2}(\Omega)$, $g\in H^{-1/2}(\partial \Omega_{N})$, and $q\in H^{1/2}(\partial \Omega_{D})$. If $\partial\Omega_{D} = \emptyset$, we further assume that $f$ and $g$ satisfy the consistency condition 
\begin{equation}
\int_{\partial \Omega} g \,d{\bm s}+ \int_{\Omega}f \,d{\bm x}=0.
\end{equation}
The weak formulation of the problem \cref{eq:1-1} is to find $u_{0}\in H^{1}_{qD}(\Omega)$ such that 
\begin{equation}\label{eq:1-2}
a(u_{0},v) = F(v),\quad \forall v\in H^{1}_{0D}(\Omega),
\end{equation}
where 
\begin{equation}\label{eq:1-2-0}
\begin{array}{lll}
{\displaystyle H^{1}_{qD}(\Omega) = \big\{v\in H^{1}(\Omega)\;:\; v = q({\bm x}) \;\;{\rm on}\;\,\partial \Omega_{D}\big\},}\\[2mm]
{\displaystyle H^{1}_{0D}(\Omega) = \big\{v\in H^{1}(\Omega)\;:\; v = 0 \;\;{\rm on}\;\,\partial \Omega_{D}\big\},}
\end{array}
\end{equation}
and the bilinear form $a(\cdot,\cdot)$ and the functional $F$ are defined by
\begin{equation}\label{eq:1-2-1}
a(u,v) = \int_{\Omega} A({\bm x})\nabla u\cdot {\nabla v} \,d{\bm x},\quad F(v) = \int_{\partial \Omega_{N}} gv \,d{\bm s}+ \int_{\Omega}fv \,d{\bm x}.
\end{equation}
For ease of notation, we define 
\begin{equation}\label{eq:1-2-2}
a_{\omega}(u,v) = \int_{\omega} A({\bm x})\nabla u\cdot {\nabla v} \,d{\bm x}, \quad F_{\omega}(v) = \int_{\partial \omega\cap\partial \Omega_{N}} gv \,d{\bm s}+ \int_{\omega}fv \,d{\bm x},
\end{equation}
%and 
%\begin{equation}\label{eq:1-2-2-0}
%F_{\omega}(v) = \int_{\partial \omega\cap\partial \Omega_{N}} gv \,d{\bm s}+ \int_{\omega}fv \,d{\bm x}
%\end{equation}
and $\Vert u \Vert_{a,\,\omega} = \sqrt{a_{\omega}(u,u)}$ for any subdomain $\omega\subset \Omega$ and $u$, $v\in H^{1}(\omega)$. If $\omega = \Omega$, the domain is omitted from the subscript and we write $a(\cdot,\cdot)$ and $\Vert\cdot\Vert_{a}$ instead of $a_{\Omega}(\cdot,\cdot)$ and $\Vert\cdot\Vert_{a,\,\Omega}$.

Under the above assumptions, in the case that $\partial\Omega_{D} \neq \emptyset$, the equation \cref{eq:1-2} has a unique solution. If $\partial\Omega_{D} = \emptyset$, the solution is unique up to an additive constant. Let $u^{p}\in H^{1}_{qD}(\Omega)$ be a {\it particular function} that satisfies the Dirichlet boundary condition and $S_{n}(\Omega)$ an $n$-dimensional subspace of $H^{1}_{0D}(\Omega)$. We seek the approximate solution of \cref{eq:1-2}, denoted by $u^{G}= u^{p} + u^{s}$, in the affine space $u^{p} + S_{n}(\Omega)$ such that
\begin{equation}\label{eq:1-2-3}
a(u^{s}, v) = F(v) - a(u^{p}, v),\quad \forall v\in S_{n}(\Omega).
\end{equation}
It is a classical result that
\begin{equation}\label{eq:1-2-4}
u^{G} = {\rm argmin}\{\Vert u_{0} - v\Vert_{a}\,:\, v\in u^{p} + S_{n}(\Omega)\},
\end{equation}
where $u_{0}$ is the solution of \cref{eq:1-2}. Therefore, if there exists a $\psi \in u^{p} + S_{n}(\Omega)$ such that $\Vert u_{0}-\psi\Vert_{a}\leq \varepsilon$, then we have $\Vert u^{G} - u_{0} \Vert_{a}\leq \varepsilon$. In what follows, we describe the construction of the particular function $u^{p}$ and of the finite dimensional space $S_{n}(\Omega)$ in the GFEM.

Let $\{ \mathcal{O}_{i} \}_{i=1}^{M}$ be a collection of open sets that cover the computational domain $\Omega$ and $\{ \chi_{i} \}_{i=1}^{M}$ be the partition of unity subordinate to the open covering. For interior sets $\mathcal{O}_{i}\subset \Omega$, we relabel them as $\omega_{i} =\mathcal{O}_{i}$ and for sets $\mathcal{O}_{i}$ that intersect the boundary of $\Omega$, we write $\omega_{i} = \mathcal{O}_{i}\cap \Omega$. Then we have $ \cup_{i=1}^{M} \omega_{i} = \Omega$. In addition, we assume that each point ${\bm x}\in\Omega$ belongs to at most $\kappa$ subdomains. The partition of unity functions are assumed to satisfy the following properties:
\begin{equation}\label{eq:1-3}
\begin{array}{lll}
{\displaystyle 0\leq \chi_{i}({\bm x})\leq 1,\quad \sum_{i=1}^{M}\chi_{i}({\bm x}) =1, \quad \forall \,{\bm x}\in \Omega,}\\[4mm]
{\displaystyle \chi_{i}({\bm x})= 0, \quad \forall \,{\bm x}\in \Omega/\omega_{i}, \quad i=1,\cdots,M,}\\[2mm]
{\displaystyle \chi_{i}\in C^{1}(\omega_{i}),\;\;\max_{{\bm x}\in\Omega} |\nabla \chi_{i}({\bm x}) | \leq \frac{C_{1}}{diam\,(\omega_{i})},\quad i=1,\cdots,M.}
\end{array}
\end{equation}

Suppose that $u^{p}_{i} \in H^{1}(\omega_{i})$ is a local particular function and $S_{n_{i}}(\omega_{i})$ is a subspace of $H^{1}(\omega_{i})$ of dimension $n_{i}$. In particular, for a subdomain $\omega_{i}$ that shares a Dirichlet boundary with $\Omega$, we require that $u^{p}_{i} = q$ on $\partial \omega_{i}\cap \partial \Omega_{D}$ and functions in $S_{n_{i}}(\omega_{i})$ vanish on $\partial \omega_{i}\cap \partial \Omega_{D}$. The global particular function $u^{p}$ and the trial space $S_{n}(\Omega)$ for the GFEM are constructed from the local particular functions and from the local approximation spaces by using the partition of unity:
\begin{equation}\label{eq:1-3-0}
\begin{array}{lll}
{\displaystyle u^{p} = \sum_{i=1}^{M}\chi_{i}u^{p}_{i},\quad  S_{n}(\Omega) =\Big\{\sum_{i=1}^{M}\chi_{i}\phi_{i}\,:\, \phi_{i}\in S_{n_{i}}(\omega_{i})\Big\}. }
\end{array}
\end{equation}
In this way, $u^{p}\in H^{1}_{qD}(\Omega)$, $S_{n}(\Omega)\subset H^{1}_{0D}(\Omega)$, and $n= \sum_{i=1}^{M}n_{i}$. 

In traditional partition of unity finite element methods, the exact solution is approximated in each subdomain and an approximation theorem \cite[Theorem 1]{babuvska1997partition} is used to estimate the global error. In this paper, instead of approximating the exact solution $u_{0}$, we approximate $\chi_{i}u_{0}$ in each subdomain $\omega_{i}$, making the global error estimate much simpler as shown in the following theorem.
\begin{theorem}\label{thm:1-0}
Assume that there exists $\phi_{i}\in S_{n_{i}}(\omega_{i})$ and $\varepsilon_{i}>0$, $i=1,\cdots,M$, such that
\begin{equation}\label{eq:1-3-1}
\big\Vert \chi_{i}(u_{0}-u^{p}_{i}-\phi_{i})\big\Vert_{a,\,\omega_{i}}\leq \varepsilon_{i}\Vert u_{0}\Vert_{a,\,\omega^{\ast}_{i}},
\end{equation}
where $\omega_{i}\subset\omega_{i}^{\ast}\subset \Omega$. Let
\begin{equation}
\Psi = u^{p} + \sum_{i=1}^{M}\chi_{i}\phi_{i}.
\end{equation}
Then $\Psi\in H^{1}_{qD}(\Omega)$ and
\begin{equation}\label{eq:1-3-2}
\begin{array}{lll}
{\displaystyle \big\Vert u_{0} - \Psi \big \Vert_{a} \leq \sqrt{\kappa\kappa^{\ast}}\big(\max_{i=1,\cdots,M}\varepsilon_{i}\big)\Vert u_{0}\Vert_{a}.}
\end{array}
\end{equation}
Here we assume that each point ${\bm x}\in\Omega$ belongs to at most $\kappa^{\ast}$ subdomains $\omega_{i}^{\ast}$. 
\end{theorem}
\begin{proof}
It is easy to show that $\Psi\in H^{1}_{qD}(\Omega)$. Moreover, using \cref{eq:1-3-1}, we have
\begin{equation}\label{eq:1-3-3}
\begin{array}{lll}
{\displaystyle \big\Vert u_{0} - \Psi \big\Vert^{2}_{a} = \Big\Vert \sum_{i=1}^{M}\chi_{i}(u_{0}-u^{p}_{i}-\phi_{i})\Big\Vert^{2}_{a}\leq \kappa \sum_{i=1}^{M}\big\Vert \chi_{i}(u_{0}-u^{p}_{i}-\phi_{i})\big\Vert^{2}_{a,\,\omega_{i}} }\\[3mm]
{\displaystyle \; \leq \kappa\sum_{i=1}^{M}\varepsilon_{i}^{2}\Vert u_{0}\Vert^{2}_{a,\,\omega^{\ast}_{i}} \leq \kappa \big(\max_{i=1,\cdots,M}\varepsilon_{i}^{2}\big)\sum_{i=1}^{M}\Vert u_{0}\Vert^{2}_{a,\,\omega^{\ast}_{i}} \leq \kappa\kappa^{\ast}\big(\max_{i=1,\cdots,M}\varepsilon_{i}^{2}\big)\Vert u_{0}\Vert^{2}_{a}, }
\end{array}
\end{equation}
which gives \cref{eq:1-3-2}. \end{proof}
It follows from \cref{eq:1-2-4} and \cref{thm:1-0} that the error of the Galerkin approximate solution $u^{G}$ is bounded by
\begin{equation}\label{eq:1-3-4}
\displaystyle \big\Vert u_{0} - u^{G} \big\Vert_{a} \leq  \big\Vert u_{0} - \Psi\big\Vert_{a}\leq \sqrt{\kappa \kappa^{\ast}}\big(\max_{i=1,\cdots,M}\varepsilon_{i}\big)\Vert u_{0}\Vert_{a}.
\end{equation}
We see that the global error of the GFEM is determined by the local approximation errors. 

In next section, we will give the local particular functions and the optimal local approximation spaces on each subdomain (\cref{thm:1-3-0,thm:3-1-0}) and derive upper bounds for the local approximation errors (\cref{thm:1-4,thm:3-2}).
\section{Local particular functions and optimal local approximation spaces}\label{sec-3}
In this section, we introduce the local particular functions and the optimal local approximation spaces for the GFEM and establish upper bounds for the local approximation errors. As in \cite{babuska2011optimal}, we decompose the solution restricted to each subdomain into two orthogonal parts. The first part satisfies the original elliptic equation locally with artificial boundary conditions on the interior boundaries, defined as the local particular function. The second part is locally $A$-harmonic. Its approximation is the key task of the MS-GFEM. We construct an optimal approximation space for the $A$-harmonic part by formulating the problem as the Kolmogorov $n$-width of a compact operator associated with the partition of unity. Due to slightly different definitions of the local particular functions and the $A$-harmonic spaces and due to some technical difficulties in the proof of the nearly exponential decay for boundary subdomains, we deal with interior subdomains and with subdomains that intersect the outer boundary separately. 
\subsection{Local approximation in interior subdomains}
In this subsection, we give the local particular function and the optimal local approximation space for a subdomain $\omega_{i}$ that lies within the interior of $\Omega$. To this end, we introduce another domain $\omega_{i}^{\ast}$ that satisfies $\omega_{i} \subseteq \omega_{i}^{\ast}\subset\Omega$ and define $\psi_{i}\in H_{0}^{1}(\omega^{\ast}_{i})$ to be the solution of
\begin{equation}\label{eq:1-4}
\left\{
\begin{array}{lll}
{\displaystyle -{\rm div} (A({\bm x})\nabla \psi_{i}({\bm x})) = f({\bm x}), \quad {\rm in }\;\,\omega^{\ast}_{i},   }\\[2mm]
{\displaystyle \psi_{i}({\bm x}) = 0,\qquad \qquad \qquad \qquad\;\; {\rm on}\;\,\partial \omega^{\ast}_{i} .}
\end{array}
\right.
\end{equation}
Next we introduce the spaces of functions that are $A$-harmonic on $\omega_{i}^{\ast}$ as follows.
\begin{equation}\label{eq:1-4-0}
\begin{array}{lll}
{\displaystyle H_{A}(\omega^{\ast}_{i}) = \big\{ v\in H^{1}(\omega^{\ast}_{i})\,:\, a_{\omega^{\ast}_{i}}(v, \varphi) = 0 \;\;\forall \varphi \in H_{0}^{1}(\omega^{\ast}_{i})\big\},}\\[2mm]
{\displaystyle H_{A,0}(\omega^{\ast}_{i}) = \big\{ v \in H_{A}(\omega^{\ast}_{i}) \,:\, \mathcal{M}_{\omega_{i}}(v) = 0\big\},}
\end{array}
\end{equation}
where 
\begin{equation}
\mathcal{M}_{\omega_{i}}(v) = \frac{\textstyle \int_{\omega_{i}} A\nabla(\chi_{i} v) \cdot\nabla \chi_{i}\,d{\bm x}} {\textstyle \int_{\omega_{i}} A\nabla\chi_{i} \cdot\nabla \chi_{i}\,d{\bm x}}
\end{equation}
with $\chi_{i}$ being the partition of unity function supported on $\omega_{i}$. It can be shown that $\Vert\cdot\Vert_{a,\omega_{i}^{\ast}}$ is a norm on $H_{A,0}(\omega^{\ast}_{i})$. From the definition of $\psi_{i}$, we see that $u_{0}|_{\omega^{\ast}_{i}} - \psi_{i}$ and $u_{0}|_{\omega^{\ast}_{i}} - \psi_{i} - \mathcal{M}_{\omega_{i}}(u_{0}|_{\omega^{\ast}_{i}} - \psi_{i})$ belong to $H_{A}(\omega^{\ast}_{i})$ and $H_{A,0}(\omega^{\ast}_{i})$, respectively, where $u_{0}$ is the solution of \cref{eq:1-2}.

Before giving the optimal approximation space, we prove an interesting identity for functions in the $A$-harmonic space which yields a Caccioppoli-type inequality \cite{babuska2011optimal} straightforwardly.
\cm{
\begin{lemma}\label{lem:1-1}
Assume that $\eta\in W^{1,\infty}(\omega_{i}^{\ast}) \cap H_{0}^{1}(\omega_{i}^{\ast})$ and $u, \,v\in H_{A}(\omega_{i}^{\ast})$. Then,  
\begin{equation}\label{eq:1-4-1}
a_{\omega_{i}^{\ast}}(\eta u,\eta v) = \int_{\omega_{i}^{\ast}} A \nabla (\eta u)\cdot \nabla (\eta v)\,d{\bm x} = \int_{\omega_{i}^{\ast}}(A\nabla \eta \cdot \nabla \eta) uv\,d{\bm x}.
\end{equation}
In particular, 
\begin{equation}\label{eq:1-4-2}
\Vert \eta u \Vert_{a, \omega_{i}^{\ast}} \leq \beta^{\frac12} \Vert \nabla \eta \Vert_{L^{\infty}(\omega_{i}^{\ast})} \Vert u \Vert_{L^{2}(\omega_{i}^{\ast})},
\end{equation}
where $\beta$ is defined in \cref{eq:1-1-0}.
\end{lemma}
\begin{proof}
A direct calculation gives 
\begin{equation}\label{eq:1-4-3}
\begin{array}{lll}
{\displaystyle a_{\omega_{i}^{\ast}}(\eta u,\eta v) = \int_{\omega_{i}^{\ast}}(A\nabla \eta \cdot \nabla \eta) uv\,d{\bm x} - \int_{\omega_{i}^{\ast}}(A\nabla u \cdot \nabla \eta) \eta v\,d{\bm x}}\\[4mm]
{\displaystyle \qquad \qquad +\,\int_{\omega_{i}^{\ast}}(A\nabla \eta \cdot \nabla v) \eta u\,d{\bm x} + \int_{\omega_{i}^{\ast}}A\nabla u \cdot \nabla (\eta^{2}v)\,d{\bm x}.}
\end{array}
\end{equation}
Since $\eta\in W^{1,\infty}(\omega_{i}^{\ast}) \cap H_{0}^{1}(\omega_{i}^{\ast})$ and $u,\,v\in H_{A}(\omega_{i}^{\ast})$, we have $\eta^{2}v\in H_{0}^{1}(\omega_{i}^{\ast})$ and thus 
\begin{equation}\label{eq:1-4-4}
\int_{\omega_{i}^{\ast}} A\nabla u\cdot\nabla (\eta^{2}v)d{\bm x} = 0.
\end{equation}
Therefore, the last term on the right-hand side of \cref{eq:1-4-3} vanishes and we get
\begin{equation}\label{eq:1-4-7}
a_{\omega_{i}^{\ast}}(\eta u,\eta v) = \int_{\omega_{i}^{\ast}}(A\nabla \eta \cdot \nabla \eta) uv\,d{\bm x} - \int_{\omega_{i}^{\ast}}(A\nabla u \cdot \nabla \eta) \eta v\,d{\bm x}
+\int_{\omega_{i}^{\ast}}(A\nabla \eta \cdot \nabla v) \eta u\,d{\bm x}.
\end{equation}
Exchanging $u$ and $v$, it follows that
\begin{equation}\label{eq:1-4-8}
a_{\omega_{i}^{\ast}}(\eta v,\eta u) = \int_{\omega_{i}^{\ast}}(A\nabla \eta \cdot \nabla \eta) uv\,d{\bm x} - \int_{\omega_{i}^{\ast}}(A\nabla v \cdot \nabla \eta) \eta u\,d{\bm x}
+\int_{\omega_{i}^{\ast}}(A\nabla \eta \cdot \nabla u) \eta v\,d{\bm x}.
\end{equation}
Now adding \cref{eq:1-4-7,eq:1-4-8} together and using the symmetry of the matrix $A$, we obtain \cref{eq:1-4-1}. \Cref{eq:1-4-2} follows immediately by taking $u=v$ in \cref{eq:1-4-1} and using \cref{eq:1-1-0}.
\end{proof}
}
In order to find the optimal approximation space for approximating a function in $H_{A,0}(\omega^{\ast}_{i})$ multiplied by the partition of unity function $\chi_{i}$, we first introduce an operator $P:H_{A,0}(\omega_{i}^{\ast})\rightarrow H_{0}^{1}(\omega_{i})$ such that $P(v)({\bm x}) = \chi_{i}({\bm x}) v({\bm x})$ for all ${\bm x}\in\omega_{i}$ and $v\in H_{A,0}(\omega_{i}^{\ast})$. Since $H^{1}(\omega_{i}^{\ast})$ is compactly embedded in $L^{2}(\omega_{i}^{\ast})$, using \cref{lem:1-1}, we have immediately that $P$ is a compact operator from $H_{A,0}(\omega_{i}^{\ast})$ into $H_{0}^{1}(\omega_{i})$. Next we consider the approximation of the set $P(H_{A,0}(\omega_{i}^{\ast}))$ in $H_{0}^{1}(\omega_{i})$ by subspaces $Q(n)$ of dimension $n$ with accuracy measured by
\begin{equation}\label{eq:1-4-6}
d(Q(n),\omega_{i}) = \sup_{u\in H_{A,0}(\omega_{i}^{\ast})} \inf_{v\in Q(n)}\frac {\Vert Pu-v\Vert_{a,\omega_{i}}}{\Vert u \Vert_{a,\omega_{i}^{\ast}}}.
\end{equation}
For each $n=1,2,\cdots$, the approximation space $\hat{Q}(n)\subset H_{0}^{1}(\omega_{i})$ is said to be optimal if it satisfies $d(\hat{Q}(n),\omega_{i}) \leq d(Q(n),\omega_{i})$ for any other $n$-dimensional space $Q(n)\subset H_{0}^{1}(\omega_{i})$. For $n=1,2,\cdots$, the problem of finding an optimal approximation space is formulated as follows. As in \cite{pinkus1985n}, the Kolmogorov $n$-width $d_{n}(\omega_{i},\omega_{i}^{\ast})$ of the compact operator $P$ is defined as
\begin{equation}\label{eq:1-4-5}
d_{n}(\omega_{i},\omega_{i}^{\ast}) = \inf_{Q(n)\subset H_{0}^{1}(\omega_{i})}d(Q(n),\omega_{i})=\inf_{Q(n)\subset H_{0}^{1}(\omega_{i})}\sup_{u\in H_{A,0}(\omega_{i}^{\ast})} \inf_{v\in Q(n)}\frac {\Vert Pu-v\Vert_{a,\omega_{i}}}{\Vert u \Vert_{a,\omega_{i}^{\ast}}}.
\end{equation}
Then the optimal approximation space $\hat{Q}(n)$ satisfies
\begin{equation}
d_{n}(\omega_{i},\omega_{i}^{\ast}) =\sup_{u\in H_{A,0}(\omega_{i}^{\ast})} \inf_{v\in \hat{Q}(n)}\frac {\Vert Pu-v\Vert_{a,\omega_{i}}}{\Vert u \Vert_{a,\omega_{i}^{\ast}}}.
\end{equation}
The $n$-width $d_{n}(\omega_{i},\omega_{i}^{\ast})$ can be characterized via the singular values and singular vectors of the compact operator $P$ as follows.
\begin{theorem}\label{thm:1-3}
For each $k\in\mathbb{N}$, let $\lambda_{k}$ and $v_{k}$ be the $k$-th  eigenvalue (arranged in increasing order) and the associated eigenfunction of the following problem 
\begin{equation}\label{eq:1-5}
a_{\omega_{i}^{\ast}}(v, \varphi) = \lambda\,a_{\omega_{i}}(\chi_{i} v, \chi_{i} \varphi) \cm{= \lambda\int_{\omega_i}(A\nabla \chi_{i} \cdot \nabla \chi_{i}) v\varphi\,d{\bm x}},\quad \forall \varphi\in H_{A,0}(\omega_{i}^{\ast}).
\end{equation}
Then the $n$-width of the compact operator $P$ satisfies $d_{n}(\omega_{i},\omega_{i}^{\ast}) =\lambda^{-1/2}_{n+1}$ and the associated optimal approximation space is given by 
\begin{equation}
\hat{Q}(n) = {\rm span}\{\chi_{i} v_{1}, \cdots, \chi_{i} v_{n}\}.
\end{equation}
\end{theorem}
\begin{proof}
Let $P^{\ast}: H_{0}^{1}(\omega_{i})\rightarrow H_{A,0}(\omega_{i}^{\ast})$ be the adjoint of the operator $P$. We denote by $\{\mu_{k}\}_{k=1}^{\infty}$, $\{v_{k}\}_{k=1}^{\infty}$, and $\{u_{k}\}_{k=1}^{\infty}$ the singular values and the right and left singular vectors of the compact operator $P$, respectively. Here $\{v_{k}\}$ are the orthonormal eigenvectors of $P^{\ast}P$ associated with the eigenvalues $\{\mu^{2}_{k}\}$, i.e.,
\begin{equation}\label{eq:1-5-0}
P^{\ast}Pv_{k} = \mu_{k}^{2} v_{k}
\end{equation} 
and $u_{k} = \mu_{k}^{-1}Pv_{k}$ for $k\in\mathbb{N}$. By \cite[Theorem 2.5]{pinkus1985n}, the $n$-width $d_{n}(\omega_{i},\omega_{i}^{\ast}) = \mu_{n+1}$ and the optimal approximation space is spanned by the left singular vectors, i.e., $\hat{Q}(n) = {\rm span} \{u_{1}, u_{2},\cdots,u_{n}\}$. Let $\lambda_{k} = \mu_{k}^{-2}$ for $k\in\mathbb{N}$. Then $d_{n}(\omega_{i},\omega_{i}^{\ast}) = \lambda^{-1/2}_{n+1}$ and the eigenvalue problem \cref{eq:1-5-0} can be written as the following variational formulation:
\begin{equation}\label{eq:1-5-1}
\begin{array}{lll}
{\displaystyle a_{\omega_{i}^{\ast}}(v_{k},\varphi) = \lambda_{k} \,a_{\omega_{i}^{\ast}}(P^{\ast}Pv_{k},\varphi) =\lambda_{k}\, a_{\omega_{i}}(Pv_{k},P\varphi)}\\[3mm]
{\displaystyle \quad  = \lambda_{k} \,a_{\omega_{i}}(\chi_{i} v_{k}, \chi_{i}\varphi)=\lambda_{k}\int_{\omega_i}(A\nabla \chi_{i} \cdot \nabla \chi_{i}) v_{k}\varphi\,d{\bm x}, \quad \forall \varphi \in H_{A,0}(\omega_{i}^{\ast}), }
\end{array}
\end{equation}
where we have used \cref{eq:1-4-1} in the last equality. We complete the proof by noting that $\hat{Q}(n) = {\rm span} \{u_{k}\}_{k=1}^{n} = {\rm span}\{\chi_{i} v_{k}\}_{k=1}^{n}$.\qquad \end{proof}
\begin{rem}
Note that thus
\vspace{-3mm}
\begin{equation}
Pu = \sum_{k=1}^{\infty}\mu_{k}a_{\omega_{i}^{\ast}}(u,v_{k})\,u_{k}
\vspace{-2mm}
\end{equation}
constitutes the singular value decomposition of the partition of unity operator $Pu = \chi_{i}u$ in the $a_{\omega_{i}^{\ast}}(\cdot,\cdot)$ inner product.
\end{rem}

With the above characterization of the $n$-width at hand, we are ready to define the optimal local approximation space on $\omega_{i}$ for the GFEM and give the local approximation error. By defining the span of the right singular vectors of the partition of unity function augmented with the space of constant functions as the local approximation space, we find the local approximation error is naturally bounded by the $n$-width.
\begin{theorem}\label{thm:1-3-0}
Let the local particular function and the optimal local approximation space on $\omega_{i}$ for the GFEM be defined as
\begin{equation}\label{eq:1-5-2}
u^{p}_{i}:=\psi_{i}|_{\omega_{i}}\;\;\; {\rm and}\;\;\; S_{n}(\omega_{i}) := {\rm span}\{v_{1}|_{\omega_{i}},\cdots, v_{n}|_{\omega_{i}}\},
\end{equation}
where $\psi_{i}$ is defined in \cref{eq:1-4} and $v_{k}$ denotes the $k$-th eigenfunction of the eigenproblem 
\begin{equation}\label{eq:1-5-3}
a_{\omega_{i}^{\ast}}(v, \varphi) = \lambda\,a_{\omega_{i}}(\chi_{i} v, \chi_{i} \varphi)\cm{=\lambda\int_{\omega_i}(A\nabla \chi_{i} \cdot \nabla \chi_{i}) v\varphi\,d{\bm x}},\quad \forall \varphi\in H_{A}(\omega_{i}^{\ast}),
\end{equation}
and let $u_{0}$ be the solution of (\ref{eq:1-2}). Then, there exists a $\phi_{i}\in S_{n}(\omega_{i})$ such that
\begin{equation}\label{eq:1-5-4}
\Vert \chi_{i}(u_{0} - u^{p}_{i} - \phi_{i})\Vert_{a,\omega_{i}}\leq d_{n-1}(\omega_{i},\omega_{i}^{\ast})\,\Vert u_{0}\Vert_{a,\omega_{i}^{\ast}}.
\end{equation}
\end{theorem}
\begin{proof}
Note that the eigenvalue problem \cref{eq:1-5-3} is posed over $H_{A}(\omega_{i}^{\ast})$ instead of $H_{A,0}(\omega_{i}^{\ast})$. First we carry out a decomposition of the local approximation space $S_{n}(\omega_{i})$. In fact, denoting by $\mathbb{R}$ the space of constant functions and recalling the definition of $H_{A,0}(\omega_{i}^{\ast})$, we observe that $H_{A}(\omega_{i}^{\ast}) = \mathbb{R} \oplus H_{A,0}(\omega_{i}^{\ast})$ and $a_{\omega_{i}^{\ast}}(v, \varphi) = a_{\omega_{i}}(\chi_{i} v, \chi_{i} \varphi) = 0$ for all $v\in \mathbb{R}$ and $\varphi \in  H_{A,0}(\omega_{i}^{\ast})$. Hence, the eigenproblem \cref{eq:1-5-3} can be decoupled into two eigenproblems: one on $\mathbb{R}$ with eigenvalue 0 and another on $H_{A,0}(\omega_{i}^{\ast})$, i.e., \cref{eq:1-5}, with positive eigenvalues. It follows that $S_{n}(\omega_{i})$ can be decomposed as
\begin{equation}\label{eq:1-5-4-1}
S_{n}(\omega_{i}) = \mathbb{R}\oplus V_{n-1}\big|_{\omega_{i}},
\end{equation}
where $V_{n-1}$ is the space spanned by the first $n-1$ eigenfunctions of \cref{eq:1-5}. 

To prove \cref{eq:1-5-4}, we first deduce from the weak formulation of \cref{eq:1-4} that 
\begin{equation}\label{eq:1-5-4-2}
a_{\omega_{i}^{\ast}}(u_{0}-\psi_{i},\varphi) = 0,\quad \forall \varphi\in H_{0}^{1}(\omega^{\ast}_{i}).
\end{equation}
Hence we have $u_{0}-\psi_{i}\in H_{A}(\omega_{i}^{\ast})$ and $u_{0}-\psi_{i}- \mathcal{M}_{\omega_{i}}(u_{0} - \psi_{i}) \in H_{A,0}(\omega_{i}^{\ast})$. Keeping in mind the definition of $V_{n-1}$ above, it follows from \cref{thm:1-3} that there exists a $\theta_{i} \in V_{n-1}\big|_{\omega_{i}}$ such that
\begin{equation}\label{eq:1-5-4-3}
\Vert \chi_{i} \big(u_{0}-\psi_{i} - \mathcal{M}_{\omega_{i}}(u_{0} - {\psi}_{i}) - \theta_{i}\big)\Vert_{a,\omega_{i}} \leq \Vert u_{0}-\psi_{i} \Vert_{a,\,\omega_{i}^{\ast}}\,d_{n-1}(\omega_{i},\omega_{i}^{\ast}).
\end{equation}
In view of \cref{eq:1-5-4-2}, we further have $\Vert u_{0}-\psi_{i} \Vert_{a,\,\omega_{i}^{\ast}} \leq \Vert u_{0}\Vert_{a,\omega_{i}^{\ast}}$. Consequently, 
\begin{equation}\label{eq:1-5-4-4}
\Vert \chi_{i} \big(u_{0}-\psi_{i} - \mathcal{M}_{\omega_{i}}(u_{0} - {\psi}_{i}) - \theta_{i}\big)\Vert_{a,\omega_{i}} \leq \Vert u_{0} \Vert_{a,\,\omega_{i}^{\ast}}\,d_{n-1}(\omega_{i},\omega_{i}^{\ast}).
\end{equation}
Define $\phi_{i}=\mathcal{M}_{\omega_{i}}(u_{0} - {\psi}_{i}) + \theta_{i}$. By \cref{eq:1-5-4-1}, we see that $\phi_{i}\in S_{n}(\omega_{i})$. The desired estimate \cref{eq:1-5-4} follows immediately from \cref{eq:1-5-4-4} and the definition of $u^{p}_{i}$.  \qquad \end{proof}
\begin{rem}\label{rem:3-1}
Note that \cref{thm:1-3,thm:1-3-0} hold for the case $\omega^{\ast}_{i}=\omega_{i}$. That is, our optimal local approximation space exists without oversampling. 
%The local approximation space $S_{n}(\omega_{i})$ for the GFEM and the optimal approximation space $\hat{Q}(n)$ of the $n$-width differ by a partition of unity function. In fact, the trial space in \cref{eq:1-3-0} for the GFEM is the sum of the optimal approximation spaces $\hat{Q}(n)$ on all the subdomains. The local approximation space $S_{n}(\omega_{i})$ is optimal in the sense that $\chi_{i}S_{n}(\omega_{i})$ is the optimal $n$-dimensional space in $H_{0}^{1}(\omega_{i})$ for approximating the function $\chi_{i}(u_{0} - u^{p}_{i})$ regardless of the constant.
\end{rem}

It remains to derive an upper bound of the local approximation error. According to \cref{thm:1-3-0}, the estimate of the local approximation error is equivalent to estimating the $n$-width $d_{n}(\omega_{i},\omega_{i}^{\ast})$. For simplicity, we assume that $\omega_{i}$ and $\omega_{i}^{\ast}$ are concentric cubes of side lengths $H_{i}$ and $H_{i}^{\ast}$ ($H_{i}^{\ast}>H_{i}$), respectively. Under this assumption, we obtain the decay rate of the $n$-width with respect to $n$ and the size of the oversampling domain as follows.
\cm{
\begin{theorem}\label{thm:1-4}
There exist $n_{0}>0$ and $b>0$, such that for any $n>n_{0}$,
\begin{equation}\label{eq:1-5-5}
d_{n}(\omega_{i},\omega_{i}^{\ast}) \leq C_{1}e^{-1}e^{-bn^{{1}/{(d+1)}}} e^{-h(\rho)bn^{{1}/{(d+1)}}},
\end{equation}
where $C_{1}$ is the constant given in \cref{eq:1-3}, $h(s) = 1+{s\log(s)}/{(1-s)}$, and $\rho = H_{i}/H_{i}^{\ast}$.
\end{theorem}
\begin{rem}\label{rem:3-3-0}
We will show in the proof of \cref{thm:1-4} that $n_{0}$ and $b$ can be explicitly chosen as 
\begin{equation}\label{eq:1-5-5-1}
n_{0} = 2(4e\Theta)^{d} \quad {\rm and}\quad b= \big(2e\Theta + 1/2\big)^{-d/(d+1)}, 
\end{equation}
where $\Theta= \frac{\gamma_{d}^{1/d}}{\sqrt{\pi}} \Big(\frac\beta\alpha\Big)^{1/2}\frac{H_{i}^{\ast}}{H_{i}^{\ast}-H_{i}}$ with $\gamma_{d}$ being the volume of the unit ball in $\mathbb{R}^{d}$.
\end{rem}
\begin{rem}
It follows from \cref{eq:1-5-5} that with $n_{\varepsilon}=\big(2e\Theta + 1/2\big)^{\frac{d}{\varepsilon (d+1)}}>0$, we can obtain a decay rate similar to the one reported in \cite{babuska2011optimal}, i.e., for any $\epsilon\in \big(0, \,\frac{1}{d+1}\big)$, if $n>n_{\varepsilon}$, then 
\begin{equation}\label{eq:1-5-5-0}
d_{n}(\omega_{i},\omega_{i}^{\ast}) \leq C_{1}e^{-1}e^{-n^{(\frac{1}{d+1}-\varepsilon)}} e^{-h(\rho)n^{(\frac{1}{d+1}-\varepsilon)}}.
\end{equation}
\end{rem}
}
\begin{rem}\label{rem:3-2}
The function $h(s)$ is plotted in \cref{fig:3.1} for $s\in (0,1)$. We observe that $h(s)$ is monotonically decreasing. In addition, $h(s)$ has the following properties:
\begin{equation}\label{eq:1-5-6}
\begin{array}{lll}
{\displaystyle 0<h(s)<1,\quad h(s)\geq\max\{0.75-s,\;\,(1-s)/2\}, \quad \forall s\in (0,1), }\\[2mm]
{\displaystyle h(s)\rightarrow 1\quad {\rm as}\;\,s\rightarrow0,\quad  h(s)\rightarrow 0\quad {\rm as}\;\,s\rightarrow1.}
\end{array}
\end{equation}
Combining \cref{eq:1-5-5-0,eq:1-5-6}, it also follows that
\begin{equation}\label{eq:1-5-7}
d_{n}(\omega_{i},\omega_{i}^{\ast})\leq C_{1}e^{-1}e^{-1.75n^{(\frac{1}{d+1}-\epsilon)}} e^{n^{(\frac{1}{d+1}-\epsilon)}H_{i}/H^{\ast}_{i}},
\end{equation}
which gives an explicit rate of convergence with respect to the size of the oversampling domain. Moreover, we see that the rate of convergence with respect to $n$ becomes higher with increasing oversampling size. If $H_{i}/H_{i}^{\ast}$ is close to 0, i.e., $\omega_{i}^{\ast}$ is much larger than $\omega_{i}$, then we can get an asymptotic convergence rate which is approximately the square of that obtained in \cite{babuska2011optimal}.
\end{rem}

\begin{figure}[!htbp]
\centering
\includegraphics[scale=0.8]{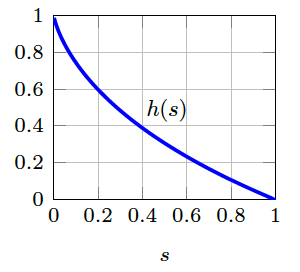}
\vspace{-3mm}\caption{The graph of the function $h(s)$ on $[0,1]$.}\label{fig:3.1}
\end{figure}

\begin{comment}
 \pgfplotsset{%
       every tick label/.append style = {font=\scriptsize},
       every axis label/.append style = {font=\scriptsize}
     }
     \begin{tikzpicture}
   \begin{axis}[grid=major, xmin=0, xmax=1.0, ymin=0, ymax=1.0,
     xlabel=$s$,
     xtick = {0,0.2,...,1}, ytick = {0,0.2,...,1},
     scale=0.4]
     \addplot[blue, samples=2000, smooth, very thick]
       plot (\x, {1+\x * ln(\x) / (1-\x)});
   \end{axis}
   \node[scale=0.8] at (1.4, 1.1) {$h(s)$};
 \end{tikzpicture}
\end{comment}

\begin{comment}
 \begin{tikzpicture}
\draw (0,0) -- (6.5,0);
\draw[red, thick] (1.0,0)--(2.5,0);
\draw[red, thick] (4.0,0)--(5.5,0);
\draw (1.5,0)--(2.0,2)--(4.5,2)--(5.0,0);
\draw (0.6,2)--(1.5,2)--(2.0,0);
\draw(4.5,0)--(5.0,2)--(5.9,2);
\draw[ultra thick] (1.0,0)--(1.0,0.1);
\draw[ultra thick] (5.5,0)--(5.5,0.1);
\draw[decoration={brace,mirror,raise=16pt},decorate]
 (1.0,0) --(5.5,0);
 \draw[decoration={brace,mirror,raise=4pt},decorate]
 (1.0,0)--(2.5,0);
 \node[below = 6pt][scale=0.8] at (1.75,0.0) {${\rm overlaps}$};
\node[below] at (3.25,-0.8) {$\omega$};

\draw[decoration={brace,raise=4pt},decorate]
 (1.0,0) --(1.5,0);
 \node[above] at (1.25,0.2) {$\delta_{\chi}$};
\node[above] at (3.25,2.2) {$\chi({\bm x})$};
\draw[->] (0,0)--(0,3);
\node[left] at (0,0) {$0$};
\node[left] at (0,2.0) {$1.0$};
\draw[thick] (0,2.0) circle [radius=0.03];
 \end{tikzpicture}
\end{comment}

The rest of this subsection is devoted to the proof of \cref{thm:1-4}. The key to the proof is to explicitly construct an $n$-dimensional subspace $Q(n)$ of $H_{0}^{1}(\omega_{i})$ such that the quantity $d(Q(n), \omega_{i})$ defined in \cref{eq:1-4-6} decays almost exponentially. To simplify the notation, the subscript index $i$ of subdomains is omitted in the proof. We first consider the following Neumann eigenvalue problem
\begin{equation}\label{eq:2-1}
a_{\omega^{\ast}}(v_{k}, \varphi) = \lambda_{k} \int_{\omega^{\ast}}v_{k}\varphi \,d{\bm x},\quad \forall \varphi \in H^{1}(\omega^{\ast}),\quad k=1,\cdots,m.
\end{equation}
Let $\Psi_{m}(\omega^{\ast})$ denote the space spanned by the first $m$ eigenfunctions of \cref{eq:2-1}. With the orthogonal decomposition
\begin{equation}\label{eq:2-1-0}
H^{1}(\omega^{\ast}) = H_{A}(\omega^{\ast})\oplus H_{0}^{1}(\omega^{\ast})
\end{equation}
with respect to the energy inner product $a_{\omega^{\ast}}(\cdot,\,\cdot)$, we further project $\Psi_{m}(\omega^{\ast})$ orthogonally from $H^{1}(\omega^{\ast})$ onto $H_{A}(\omega^{\ast})$ and denote it by $W_{m}(\omega^{\ast}) = \mathcal{P}^{A}\Psi_{m}(\omega^{\ast})$. Using the classical Weyl asymptotics for the Laplacian and the comparison principle for eigenvalue problems, we can prove the following approximation property in $W_{m}(\omega^{\ast})$.
\begin{lemma}\label{lem:2-1}
For any $u\in H_{A}(\omega^{\ast})$, there exists a $v_{u}\in W_{m}(\omega^{\ast})$ such that
\begin{equation}\label{eq:2-0-0}
\Vert u - v_{u}\Vert_{L^{2}(\omega^{\ast})} = \inf_{v\in W_{m}(\omega^{\ast})} \Vert u - v\Vert_{L^{2}(\omega^{\ast})} \leq C(m)H^{\ast}\frac{\gamma_{d}^{1/d}}{\sqrt{4\pi}}\alpha^{-1/2}\Vert u \Vert_{a,\omega^{\ast}},
\end{equation}
where $H^{\ast}$ is the side length of the cube $\omega^{\ast}$, $\gamma_{d}$ is the volume of the unit ball in $\mathbb{R}^{d}$, and $C(m) = m^{-1/d}(1+o(1))$.
\end{lemma}
A similar lemma was proved in \cite{babuska2011optimal}. The proof is given in \cref{sec:A.3} for completeness. \Cref{lem:2-1} gives the approximation error in the $L^{2}$ norm. Combining \cref{lem:1-1,lem:2-1}, we are able to give an approximation property in the energy norm.
\begin{lemma}\label{lem:2-2}
Assume that $\eta\in W^{1,\infty}(\omega^{\ast})\cap H_{0}^{1}(\omega^{\ast})$. For any $u\in H_{A}(\omega^{\ast})$, there exists a $v_{u}\in W_{m}(\omega^{\ast})$ such that
\begin{equation}
\Vert \eta(u - v_{u})\Vert_{a,\,\omega^{\ast}} \leq C(m)H^{\ast}\frac{\gamma_{d}^{1/d}}{\sqrt{4\pi}}\Big(\frac{\beta}{\alpha}\Big)^{1/2}\Vert \nabla \eta\Vert_{L^{\infty}(\omega^{\ast})}\Vert u \Vert_{a,\omega^{\ast}}.
\end{equation}
\end{lemma}
\begin{rem}\label{rem:3-4}
If we choose $\eta = \chi$ in \cref{lem:2-2}, where $\chi$ is the partition of unity function supported on $\omega$, then we get
\begin{equation}\label{eq:2-1-1}
\begin{array}{lll}
{\displaystyle d_{m}(\omega, \omega^{\ast})\leq \sup_{u\in H_{A,0}(\omega^{\ast})}\inf_{v\in W_{m}(\omega^{\ast})}\frac{\Vert \chi(u-v)\Vert_{a,\omega}}{\Vert u \Vert_{a,\omega^{\ast}}}}\\[2mm]
{\displaystyle \qquad\qquad\,\leq C_{1}C(m)\frac{\gamma_{d}^{1/d}}{\sqrt{4\pi}}\Big(\frac{\beta}{\alpha}\Big)^{1/2}\frac{H^{\ast}}{H},}
\end{array}
\end{equation}
where $C_{1}$ is the constant introduced in \cref{eq:1-3}. Note that \cref{eq:2-1-1} holds for $H=H^{\ast}$, i.e., $\omega=\omega^{\ast}$. Therefore, our method converges even without oversampling, which does not hold for the optimal local approximation introduced in \cite{babuska2011optimal}. \cm{However, without oversampling the decay rate of the method is not exponential with respect to the dimension of the local spaces}.
\end{rem}

Based on \cref{lem:2-2}, we now proceed to define a new space for approximating any $u\in H_{A}(\omega^{\ast})$ with a higher convergence rate. Let $N\geq 1$ be an integer. We choose $\omega^{j}$, $j=1,2,\cdots,N,N+1$, to be the nested family of concentric cubes with side length $H^{\ast}-\delta^{\ast} (j-1)/N$ for which $\omega=\omega^{N+1}\subset \omega^{N}\subset\cdots\subset\omega^{1} = \omega^{\ast}$, where $\delta^{\ast} = H^{\ast}-H$. Note that ${dist}\,(\omega^{k},\omega^{k+1}) = \delta^{\ast}/(2N)$. Let $n=(N+1)\times m$. We define 
\begin{equation}\label{eq:2-8}
\mathcal{T}(n,\omega,\omega^{\ast}) = W_{m}(\omega^{1})+ \cdots+W_{m}(\omega^{N+1}).
\end{equation}
We have the following convergence rate for the approximation space $\mathcal{T}(n,\omega, \omega^{\ast})$.
\begin{lemma}\label{thm:2-1}
Let $u\in H_{A}(\omega^{\ast})$ and $N\geq 1$ be an integer. Then there exists a $z_{u}\in \mathcal{T}(n,\omega,\omega^{\ast})$ such that
\begin{equation}\label{eq:2-9}
\Vert \chi(u -z_{u}) \Vert_{a,\omega} \leq \frac{C_{1}}{2\sqrt{2}N}\prod_{k=1}^{N-1} \big(1-\frac{k\delta^{\ast}}{NH^{\ast}}\big) \xi^{N+1} \Vert u \Vert_{a,\omega^{\ast}},
\end{equation}
where $C_{1}$ is the (positive) constant introduced in \cref{eq:1-3} and $\xi$ is given by
\begin{equation}\label{eq:2-11}
\xi = \xi(N,m)=N\frac{\gamma_{d}^{1/d}}{\sqrt{\pi}}\Big(\frac{\beta}{\alpha}\Big)^{1/2} \frac{H^{\ast}}{\delta^{\ast}}C(m).
\end{equation}
\end{lemma}
\begin{proof}
We begin by introducing a family of cut-off functions $\eta_{k}\in C_{0}^{1}(\omega^{k})$, $k=1,2,\cdots,N$, such that $\eta_{k}({\bm x})=1$ for ${\bm x}\in \omega^{k+1}$ and $|\nabla \eta_{k}({\bm x})| \leq 2N/\delta^{\ast}$. For $u\in H_{A}(\omega^{\ast}) = H_{A}(\omega^{1})$, it follows from \cref{lem:2-2} that there exists a $v_{u}^{1}\in W_{m}(\omega^{1})$ such that
\begin{equation}\label{eq:2-10}
\Vert \eta_{1}(u - v_{u}^{1}) \Vert_{a,\omega^{1}} \leq \xi \Vert u \Vert_{a,\omega^{1}}.
\end{equation}
Note that $u-v_{u}^{1} \in H_{A}(\omega^{2})$ and the side length of $\omega^{2} = H^{\ast}-\delta^{\ast}/N$. Applying \cref{lem:2-2} again, we can find a $v_{u}^{2}\in W_{m}(\omega^{2})$ such that
\begin{equation}\label{eq:2-12}
\begin{array}{lll}
{\displaystyle  \Vert \eta_{2}(u-v_{u}^{1} - v_{u}^{2}) \Vert_{a,\omega^{2}} \leq \big(1-\frac{\delta^{\ast}}{NH^{\ast}}\big)\xi \Vert u-v_{u}^{1} \Vert_{a,\omega^{2}} }\\[2mm]
{\displaystyle  \leq \big(1-\frac{\delta^{\ast}}{NH^{\ast}}\big)\xi \Vert \eta_{1}(u-v_{u}^{1}) \Vert_{a,\omega^{1}} \leq \big(1-\frac{\delta^{\ast}}{NH^{\ast}}\big)\xi^{2} \Vert u \Vert_{a,\omega^{1}},}
\end{array}
\end{equation}
where we have used \cref{eq:2-10} in the last inequality. Repeating this process until $k=N$, we see that there exists a $v_{u}^{N}\in W_{m}(\omega^{N})$ such that
\begin{equation}\label{eq:2-13}
\Vert \eta_{N}(u - \sum_{k=1}^{N-1}v_{u}^{k} - v_{u}^{N} ) \Vert_{a,\omega^{N}} \leq \prod_{k=1}^{N-1} \big(1-\frac{k\delta^{\ast}}{NH^{\ast}}\big)\xi^{N} \Vert u \Vert_{a,\omega^{1}}.
\end{equation}
Finally, applying \cref{lem:2-2} with $\eta = \chi$, we can find a $v_{u}^{N+1}\in W_{m}(\omega^{N+1})$ such that
\begin{equation}\label{eq:2-14}
\begin{array}{lll}
{\displaystyle  \Vert \chi(u - \sum_{k=1}^{N}v_{u}^{k} -  v_{u}^{N+1}) \Vert_{a,\omega}  =\Vert \chi(u - \sum_{k=1}^{N}v_{u}^{k} - v_{u}^{N+1}) \Vert_{a,\omega^{N+1}} }\\[3mm]
{\displaystyle \leq \frac{C_{1}\xi}{2\sqrt{2}N}\Vert \eta_{N}(u - \sum_{k=1}^{N-1}v_{u}^{k} - v_{u}^{N} ) \Vert_{a,\omega^{N}} \leq  \frac{C_{1}}{2\sqrt{2}N}\prod_{k=1}^{N-1} \big(1-\frac{k\delta^{\ast}}{NH^{\ast}}\big) \xi^{N+1} \Vert u \Vert_{a,\omega^{\ast}},}
\end{array}
\end{equation}
where we have used the fact that $|\nabla \chi({\bm x})| \leq C_{1}/diam\,(\omega)$. Therefore, $z_{u} = \sum_{k=1}^{N+1}v_{u}^{k}$ satisfies \cref{eq:2-9}.\qquad \end{proof}

\cm{
{\it Proof of \cref{thm:1-4}}. In what follows, we show how \cref{thm:2-1} leads to the desired estimate \cref{eq:1-5-5}. To this end, we choose $m$ such that 
\begin{equation}\label{eq:2-14-0}
\frac{\Theta(N+1)^{2}e}{N}\leq m^{1/d}<\frac{\Theta(N+1)^{2}e}{N} +1,
\end{equation}
where $\Theta=\frac{\gamma_{d}^{1/d}}{\sqrt{\pi}}\big(\frac{\beta}{\alpha}\big)^{1/2} \frac{H^{\ast}}{\delta^{\ast}}$. From \cref{eq:2-14-0,eq:2-11}, we deduce that
\begin{equation}
\xi\leq \frac{1}{e}\frac{N^{2}}{(N+1)^{2}},
\end{equation}
and consequently
\begin{equation}\label{eq:2-15}
\xi^{N+1}\leq e^{-(N+1)}\Big(\frac{N}{N+1}\Big)^{2(N+1)}\leq e^{-2}e^{-(N+1)}.
\end{equation}
Furthermore, it can be proved that
\begin{equation}\label{eq:2-15-0}
\prod_{k=1}^{N-1} \big(1-\frac{k\delta^{\ast}}{NH^{\ast}}\big) \leq e\sqrt{N}e^{-Nh(\rho)},
\end{equation}
where $h(s) = 1+{s\log(s)}/{(1-s)}$ and $\rho =1- \delta^{\ast}/H^{\ast} = H/H^{\ast}$. The proof of \cref{eq:2-15-0} is given in \cref{lem:2-3} at the end of this subsection. Combining \cref{eq:2-15,eq:2-15-0}, we have
\begin{equation}\label{eq:2-15-1}
\begin{array}{lll}
{\displaystyle \frac{C_{1}}{2\sqrt{2}N}\prod_{k=1}^{N-1} \big(1-\frac{k\delta^{\ast}}{NH^{\ast}}\big) \xi^{N+1}\leq \frac{C_{1}e^{-1}}{2\sqrt{2N}} e^{-(N+1)}e^{-Nh(\rho)}}\\[4mm]
{\displaystyle \qquad  \leq C_{1}e^{-1}e^{-(N+1)}e^{-(N+1)h(\rho)},}
\end{array}
\end{equation}
where we have used the fact that $0<h(\rho)<1$. By \cref{eq:2-14-0}, we have
\begin{equation}
m^{1/d}<\frac{\Theta(N+1)^{2}e}{N} +1\leq (2e\Theta+1/2)(N+1),
\end{equation}
which implies that $n=(N+1)m \leq  (2e\Theta+1/2)^{d}(N+1)^{d+1}$ and thus 
\begin{equation}\label{eq:2-15-2}
N+1 \geq  (2e\Theta+1/2)^{-d/(d+1)}\,n^{1/(d+1)}.
\end{equation}
Combining \cref{eq:2-15-1,eq:2-15-2} gives
\begin{equation}\label{eq:2-16}
\frac{C_{1}}{2\sqrt{2}N}\prod_{k=1}^{N-1} \big(1-\frac{k\delta^{\ast}}{NH^{\ast}}\big) \xi^{N+1}\leq C_{1}e^{-1}e^{-bn^{{1}/{(d+1)}}} e^{-h(\rho)bn^{{1}/{(d+1)}}},
\end{equation}
where $b = (2e\Theta+1/2)^{-d/(d+1)}$. In view of \cref{eq:2-14-0} and the fact that $N\geq 1$, we find
\begin{equation}\label{eq:2-18}
n=(N+1)m\geq (N+1)\big(e\Theta(N+1)^{2}/N\big)^{d}\geq 2(4e\Theta)^{d}.
\end{equation}
Finally, we define an $n$-dimensional subspace of $H^{1}_{0}(\omega)$ as $Q(n) = \chi \,\mathcal{T}(n,\omega,\omega^{\ast})$. \Cref{thm:2-1} together with \cref{eq:2-16,eq:2-18} yield that
\begin{equation}\label{eq:2-19}
\begin{array}{lll}
{\displaystyle d_{n}(\omega,\omega^{\ast}) \leq \sup_{u\in H_{A,0}(\omega^{\ast})} \inf_{v\in Q(n)}\frac {\Vert \chi u-v\Vert_{a,\omega}}{\Vert u \Vert_{a,\omega^{\ast}}} }\\[4mm]
{\displaystyle \qquad \qquad\; \;\leq C_{1}e^{-1}e^{-bn^{{1}/{(d+1)}}} e^{-h(\rho)bn^{{1}/{(d+1)}}},}
\end{array}
\end{equation}
if $n\geq n_{0}=2(4e\Theta)^{d}$. This completes the proof of \cref{thm:1-4}.}\qquad\qquad $\square$

We conclude this subsection by proving the auxiliary inequality \cref{eq:2-15-0} used in the proof of \cref{thm:1-4}.
\begin{lemma}\label{lem:2-3}
Let $N\geq 2$ be an integer and $H^{\ast}>0$. Then the inequality
\begin{equation}\label{eq:2-20}
 \prod_{k=1}^{N-1} \big(1-\frac{k\delta^{\ast}}{NH^{\ast}}\big) \leq e\sqrt{N}e^{-Nh(\rho)}
\end{equation}
holds for any $\delta^{\ast}\in (0,H^{\ast})$, where $h(s) = 1+{s\log(s)}/{(1-s)}$ and $\rho =1- \delta^{\ast}/H^{\ast}$.
\end{lemma}
\begin{proof}
Let $\sigma = \delta^{\ast}/H^{\ast}$. Taking the natural logarithm of both sides of \cref{eq:2-20}, it suffices to prove
\begin{equation}\label{eq:2-21}
 \sum_{k=1}^{N-1} \log{(1-\frac{k}{N}\sigma)} \leq 1+\frac{1}{2}\log{N}-Nh(1-\sigma),\quad \forall \sigma\in (0,1).
\end{equation}
Introduce the function
\begin{equation}
\begin{array}{lll}
{\displaystyle G(\sigma) = \sum_{k=1}^{N-1} \log{(1-\frac{k}{N}\sigma)} - 1 - \frac{1}{2}\log{N}+Nh(1-\sigma)}\\[2mm]
{\displaystyle = \sum_{k=1}^{N-1} \log{(1-\frac{k}{N}\sigma)} - 1 - \frac{1}{2}\log{N}+N\big(1+(1-\sigma)\log{(1-\sigma)}/\sigma\big).}
\end{array}
\end{equation}
It is easy to see that $\displaystyle \lim_{\sigma\rightarrow 1}G(\sigma) = \log{(N!/N^{N})} - 1 - \frac{1}{2}\log{N}+N$. Using the classical inequality $N!\leq e N^{N+1/2}e^{-N}$, we find $\displaystyle \lim_{\sigma\rightarrow 1}G(\sigma)\leq 0$. Hence, to prove \cref{eq:2-21}, we only need to show that $G(\sigma)$ is monotonically increasing on $(0,1)$. Now taking the derivative of $G(\sigma)$ and using the Taylor series expansions of functions $1/(1-x)$ and $\log(1-x)$, we find $G^{\prime}(\sigma)>0$, which completes the proof of this lemma. \qquad \end{proof}
\subsection{Local approximation at the boundary}

%\begin{figure}[!htbp]
%\centering
% \begin{tikzpicture} 
% \draw[pattern=north east hatch, hatch distance=2.5mm, hatch thickness=.5pt, pattern color=red] (13, 0.75)--(12.5, 0.75)--(12.5,-0.75)--(13,-0.75)--(14,0)--(13,0.75); 
% \draw[pattern=dots, pattern color=green] (12.5, 1.125)--(11.8, 1.125)--(11.8, -1.125)--(12.5,-1.125)--(14,0)--(12.5, 1.125);
% 
%\draw[thick] (11.8,-1.65)--(14,0)--(11.8,1.65);
%\draw[thick] (13, 0.75)--(12.5, 0.75)--(12.5,-0.75)--(13,-0.75);
%\draw[thick] (12.5, 1.125)--(11.8, 1.125)--(11.8, -1.125)--(12.5,-1.125);
%
%\draw[->, dashed, very thick] (11.3,-1.1)--(12.5,0);
%\draw[->, dashed, very thick] (10.8,0.2)--(11.8,0.8);
%   
%\node[scale=1.8] at (10.9,-1.3) {$\omega_{i}$};
%\node[scale=1.8] at (10.3,0) {$\omega^{\ast}_{i}$};
%\node at (13.5,0.85) {$\partial \Omega_{D}$};
%\node at (13,-1.3) {$\partial \Omega_{N}$};
% \end{tikzpicture}
%\caption{Illustration of a subdomain that touches the boundary of $\Omega$ and the associated oversampling domain.}\label{fig:3.2}
%\end{figure}
 
In this subsection, we introduce the local particular function and the optimal local approximation space for a subdomain $\omega_{i}$ that intersects the boundary of $\Omega$. As before, we introduce another domain $\omega_{i}^{\ast}$ such that $\omega_{i}\subseteq \omega_{i}^{\ast}\subset\Omega$ as illustrated in \cref{fig:1-1}. Without loss of generality, we assume that $\partial \omega^{\ast}_{i} \bigcap \partial \Omega_{D} \neq \emptyset$. The pure Neumann boundary case can be addressed in a similar way. Following the ideas of \cite{babuvska2020multiscale}, we first define a function $\psi_{i} = \psi_{i}^{r} + \psi_{i}^{d}$, where $\psi_{i}^{r}$ and $\psi_{i}^{d}$ satisfy
\begin{equation}\label{eq:3-1}
\left\{
\begin{array}{lll}
{\displaystyle -{\rm div} (A({\bm x})\nabla \psi^{r}_{i}({\bm x})) = f({\bm x}), \quad {\rm in}\;\, \omega^{\ast}_{i}  }\\[2mm]
{\displaystyle {\bm n}\cdot A({\bm x})\nabla\psi^{r}_{i}({\bm x}) = g({\bm x}),\qquad \;\; {\rm on}\;\, \partial \omega^{\ast}_{i} \cap \partial \Omega_{N}}\\[2mm]
{\displaystyle \psi^{r}_{i}({\bm x}) = 0,\quad \quad \quad \qquad \qquad \quad \;\;{\rm on}\;\, \partial \omega^{\ast}_{i}\cap \Omega}\\[2mm]
{\displaystyle \psi^{r}_{i}({\bm x}) = 0,\;\quad \quad \qquad \qquad \qquad \;{\rm on}\;\, \partial \omega^{\ast}_{i} \cap \partial \Omega_{D}}
\end{array}
\right.
\end{equation}
and 
\begin{equation}\label{eq:3-1-0}
\left\{
\begin{array}{lll}
{\displaystyle -{\rm div} (A({\bm x})\nabla \psi^{d}_{i}({\bm x})) = 0, \quad \,{\rm in} \;\,\omega^{\ast}_{i}  }\\[2mm]
{\displaystyle {\bm n}\cdot A({\bm x})\nabla\psi^{d}_{i}({\bm x}) = 0,\qquad \;\;\, {\rm on}\;\, \partial \omega^{\ast}_{i} \cap \partial \Omega_{N}}\\[2mm]
{\displaystyle {\bm n}\cdot A({\bm x})\nabla\psi^{d}_{i}({\bm x}) = 0,\quad \quad \;\; \;{\rm on} \;\,\partial \omega^{\ast}_{i}\cap \Omega}\\[2mm]
{\displaystyle \psi^{d}_{i}({\bm x}) = q({\bm x}),\;\;\;\quad \qquad \qquad {\rm on} \;\,\partial \omega^{\ast}_{i} \cap \partial \Omega_{D},}
\end{array}
\right.
\end{equation}
respectively. \cm{A similar treatment of the mixed boundary conditions was also discussed in \cite{chen2021exponential}}. By definition, we see that $(u_{0}-\psi_{i})({\bm x}) = 0$ on $\partial \omega^{\ast}_{i} \cap \partial \Omega_{D}$, where $u_{0}$ is the solution of \cref{eq:1-2}. Moreover, it can be proved that
\begin{equation}\label{eq:3-1-1}
a_{\omega^{\ast}_{i}}(u_{0}-\psi_{i}, v) = 0, \quad \forall v\in H^{1}_{0D}(\omega^{\ast}_{i}),
\end{equation}
where
\begin{equation}\label{eq:3-1-2}
H^{1}_{0D}(\omega^{\ast}_{i}) = \big\{v\in H^{1}(\omega^{\ast}_{i})\;:\;v = 0\;\;{\rm on}\;\,  (\partial \omega^{\ast}_{i}\cap \Omega)\cup(\partial\omega^{\ast}_{i} \cap \partial \Omega_{D})\big\}.
\end{equation}
In fact, the weak formulations of \cref{eq:3-1,eq:3-1-0} imply that
\begin{equation}\label{eq:3-1-3}
a_{\omega^{\ast}_{i}}(\psi_{i}, v) = F_{\omega_{i}^{\ast}}(v), \quad \forall v\in H^{1}_{0D}(\omega^{\ast}_{i}),
\end{equation}
where $F_{\omega_{i}^{\ast}}(\cdot)$ is defined in (\ref{eq:1-2-2}). A combination of \cref{eq:1-2,eq:3-1-3} gives \cref{eq:3-1-1}. Define
\begin{equation}\label{eq:3-2}
\begin{array}{lll}
{\displaystyle  H_{A,D}(\omega^{\ast}_{i}) = \big\{ u\in H^{1}(\omega^{\ast}_{i})\;:\; a_{\omega^{\ast}_{i}}(u, v) = 0, \quad \forall v\in H^{1}_{0D}(\omega^{\ast}_{i})\big\}, }\\[2mm]
{\displaystyle H^{0}_{A,D}(\omega^{\ast}_{i}) = \big\{ v\in H_{A,D}(\omega^{\ast}_{i})\;:\;v = 0 \;\;{\rm on} \;\,\partial \omega^{\ast}_{i} \cap \partial \Omega_{D}\big\}.}
\end{array}
\end{equation}
We see that $u_{0}-\psi_{i}\in H^{0}_{A,D}(\omega^{\ast}_{i})$. 
\begin{rem}
In \cite{babuska2011optimal}, the $A$-harmonic spaces on boundary subdomains are defined in the same way as for interior subdomains in which functions are $a$-orthogonal to $H_{0}^{1}(\omega_{i}^{\ast})$. In this paper, we take the boundary conditions into account and introduce the different $A$-harmonic spaces on boundary subdomains in which functions are $a$-orthogonal to a bigger space $H_{0D}^{1}(\omega_{i}^{\ast})$. This facilitates our subsequent analysis.
\end{rem}

In what follows, we proceed in the same way as for interior subdomains to introduce the optimal approximation space for approximating a function in $H^{0}_{A,D}(\omega^{\ast}_{i})$ multiplied by the partition of unity function. The following lemma is the counterpart of \cref{lem:1-1} for boundary subdomains. It can be proved by using the fact that $\eta^{2}u\in H^{1}_{0D}(\omega_{i}^{\ast})$ and a similar argument as in the proof of \cref{lem:1-1}. 
\cm{
\begin{lemma}\label{lem:3-1}
Assume that $\eta \in W^{1,\infty}(\omega_{i}^{\ast})$ satisfies $\eta({\bm x}) = 0$ on $\partial\omega_{i}^{\ast} \cap \Omega$. Then, for any $u,\,v\in H^{0}_{A,D}(\omega_{i}^{\ast})$,
\begin{equation}\label{eq:3-4}
\int_{\omega_{i}^{\ast}} A \nabla (\eta u)\cdot \nabla (\eta v)\,d{\bm x} = \int_{\omega_{i}^{\ast}}(A\nabla \eta \cdot \nabla \eta) uv\,d{\bm x}.
\end{equation}
In particular,
\begin{equation}\label{eq:3-5}
\Vert \eta u \Vert_{a, \omega_{i}^{\ast}} \leq \beta^{\frac12} \Vert \nabla \eta \Vert_{L^{\infty}(\omega_{i}^{\ast})} \Vert u \Vert_{L^{2}(\omega_{i}^{\ast})},\quad \forall u\in H^{0}_{A,D}(\omega_{i}^{\ast}).
\end{equation}
\end{lemma}
}

Now we introduce the operator $P:H^{0}_{A,D}(\omega_{i}^{\ast})\rightarrow H^{1}_{0D}(\omega_{i})$ such that $P(u)({\bm x}) = \chi_{i}({\bm x}) u({\bm x})$ for all ${\bm x}\in\omega_{i}$ and $u\in H^{0}_{A,D}(\omega_{i}^{\ast})$. Using \cref{lem:3-1} and the Rellich compactness theorem, we find that the operator $P$ is compact from $H^{0}_{A,D}(\omega_{i}^{\ast})$ into $H_{0D}^{1}(\omega_{i})$. As before, we consider approximating the set $P(H^{0}_{A,D}(\omega_{i}^{\ast}))$ in $H_{0D}^{1}(\omega_{i})$ by $n$-dimensional subspaces $Q(n)$. For $n\in\mathbb{N}$, the problem of finding the optimal approximation space is formulated as follows. Let
\begin{equation}\label{eq:3-12}
d_{n}(\omega_{i},\omega_{i}^{\ast}) = \inf_{Q(n)\subset H_{0D}^{1}(\omega_{i})}\sup_{u\in H^{0}_{A,D}(\omega_{i}^{\ast})} \inf_{v\in Q(n)}\frac {\Vert Pu-v\Vert_{a,\omega_{i}}}{\Vert u \Vert_{a,\omega_{i}^{\ast}}}.
\end{equation}
The optimal approximation space $\hat{Q}(n)$ satisfies
\begin{equation}\label{eq:3-13}
d_{n}(\omega_{i}, \omega_{i}^{\ast}) =\sup_{u\in H^{0}_{A,D}(\omega_{i}^{\ast})} \inf_{v\in \hat{Q}(n)}\frac {\Vert Pu-v\Vert_{a,\omega_{i}}}{\Vert u \Vert_{a,\omega_{i}^{\ast}}}.
\end{equation}
As for interior subdomains, the $n$-width $d_{n}(\omega_{i},\omega_{i}^{\ast})$ can be characterized as follows.
\begin{theorem}\label{thm:3-1}
For each $k\in\mathbb{N}$, let $\lambda_{k}$ and $v_{k}$ be the $k$-th eigenvalue and the associated eigenfunction of the following problem
\begin{equation}\label{eq:3-14}
a_{\omega_{i}^{\ast}}(v, \varphi)=\lambda\,a_{\omega_{i}}(\chi_{i} v, \chi_{i} \varphi)\cm{=\lambda\int_{\omega_i}(A\nabla \chi_{i} \cdot \nabla \chi_{i}) v\varphi\,d{\bm x}},\quad \forall \varphi\in H_{A,D}^{0}(\omega_{i}^{\ast}).
\end{equation}
Then, $d_{n}(\omega_{i}, \omega_{i}^{\ast}) = \lambda^{-1/2}_{n+1}$, and the associated optimal approximation space is given by $\hat{Q}(n) = {\rm span}\{\chi_{i}v_{1}, \cdots, \chi_{i} v_{n}\}$.
\end{theorem}
\begin{rem}
If the domain $\omega_{i}^{\ast}$ only shares a Neumann boundary with $\Omega$, i.e., $\partial \omega_{i}^{\ast}\cap\partial \Omega_{D} =\emptyset$,  then we work on the spaces 
\begin{equation}\label{eq:3-15}
\begin{array}{lll}
{\displaystyle H^{1}_{0N}(\omega^{\ast}_{i}) = \big\{v\in H^{1}(\omega^{\ast}_{i})\;:\;v = 0\;\;{\rm on}\;\,  \partial \omega^{\ast}_{i}\cap \Omega\big\},}\\[2mm]
{\displaystyle H_{A,N}(\omega^{\ast}_{i}) = \big\{ u\in H^{1}(\omega^{\ast}_{i})\;:\; a_{\omega^{\ast}_{i}}(u, v) = 0, \quad \forall v\in H^{1}_{0N}(\omega^{\ast}_{i}) \big\}, }\\[2mm]
{\displaystyle H^{0}_{A,N}(\omega_{i}^{\ast}) = \big\{ u\in H_{A,N}(\omega_{i}^{\ast})\;:\;\mathcal{M}_{\omega_{i}} (u) = 0 \big\}, }
\end{array}
\end{equation}
and \cref{thm:3-1} still holds for this case with $H_{A,D}^{0}(\omega_{i}^{\ast})$ replaced by $H^{0}_{A,N}(\omega_{i}^{\ast})$.
\end{rem}

Proceeding as before, we define the local particular function and the optimal local approximation space on a subdomain $\omega_{i}$ that touches the boundary of $\Omega$.

\begin{theorem}\label{thm:3-1-0}
Let the local particular function and the optimal local approximation space on $\omega_{i}$ for the GFEM be defined as
\begin{equation}
u^{p}_{i}:=\psi_{i}|_{\omega_{i}}\;\;\;{\rm and}\;\;\quad S_{n}(\omega_{i}) := {\rm span}\{v_{1}|_{\omega_{i}},\cdots, v_{n}|_{\omega_{i}}\},
\end{equation}
where $\psi_{i} = \psi_{i}^{r} + \psi_{i}^{d}$ with $\psi_{i}^{r}$ and $\psi_{i}^{d}$ defined in \cref{eq:3-1,eq:3-1-0}, respectively, and $v_{k}$ is the $k$-th eigenfunction of the eigenproblem
\begin{equation}\label{eq:3-15-1}
a_{\omega_{i}^{\ast}}(v, \varphi) = \lambda\,a_{\omega_{i}}(\chi_{i} v, \chi_{i} \varphi)\cm{=\lambda\int_{\omega_i}(A\nabla \chi_{i} \cdot \nabla \chi_{i}) v\varphi\,d{\bm x}},\quad \forall \varphi\in H_{A,D}^{0}(\omega_{i}^{\ast}),
\end{equation}
and let $u_{0}$ be the solution of \cref{eq:1-2}. Then, there exists a $\phi_{i}\in S_{n}(\omega_{i})$ such that
\begin{equation}\label{eq:3-15-2}
\Vert \chi_{i}(u_{0} - u^{p}_{i} - \phi_{i})\Vert_{a,\omega_{i}}\leq d_{n}(\omega_{i},\omega_{i}^{\ast})\,\Vert u_{0}\Vert_{a,\omega_{i}^{\ast}}\quad\; {\rm if}\;\, \partial \omega_{i}^{\ast}\cap \partial \Omega_{D} \neq \emptyset,
\end{equation}
or
\begin{equation}\label{eq:3-15-3}
\Vert \chi_{i}(u_{0} - u^{p}_{i} - \phi_{i})\Vert_{a,\omega_{i}}\leq d_{n-1}(\omega_{i},\omega_{i}^{\ast})\,\Vert u_{0}\Vert_{a,\omega_{i}^{\ast}}\quad\; {\rm if}\;\, \partial \omega_{i}^{\ast}\cap \partial \Omega_{D} = \emptyset.
\end{equation}
\end{theorem}
\begin{rem}
We construct the local approximation space using eigenfunctions of the eigenproblem \cref{eq:3-15-1} for both the case $\partial \omega_{i}^{\ast}\cap \partial \Omega_{D} \neq \emptyset$ and $\partial \omega_{i}^{\ast}\cap \partial \Omega_{D} = \emptyset$, because in the latter case, the space $H_{A,D}^{0}(\omega_{i}^{\ast})$ reduces to $H_{A,N}(\omega^{\ast}_{i})$ defined in \cref{eq:3-15}. The difference of the local approximation errors in \cref{eq:3-15-2,eq:3-15-3} arises since the local approximation space needs to be augmented with the space of constant functions when $\partial \omega_{i}^{\ast}\cap \partial \Omega_{D} = \emptyset$ as for interior subdomains.
\end{rem}

\begin{proof}
Inequalities \cref{eq:3-15-2,eq:3-15-3} can be proved by using a similar argument as in the proof of \cref{thm:1-3-0} and the following inequality
\begin{equation}\label{eq:3-15-4}
\Vert u_{0}-\psi_{i}\Vert_{a,\omega_{i}^{\ast}}\leq \Vert u_{0}\Vert_{a,\omega_{i}^{\ast}}.
\end{equation}
To prove \cref{eq:3-15-4}, we first observe that by definition, it holds that $u_{0}-\psi_{i}\in H_{A,D}(\omega_{i}^{\ast})$ and $\psi_{i}^{r}\in H^{1}_{0D}(\omega^{\ast}_{i})$. Consequently, we have $a_{\omega_{i}^{\ast}}(u_{0}-\psi_{i}, \psi_{i}^{r}) = 0$. In addition, noting that $u_{0}-\psi_{i}$ vanishes on $\partial \omega^{\ast}_{i} \cap \partial \Omega_{D}$, the weak formulation of \cref{eq:3-1-0} implies that $a_{\omega_{i}^{\ast}}(u_{0}-\psi_{i}, \psi_{i}^{d}) = 0$. Hence,
\begin{equation}
\begin{array}{lll}
{\displaystyle a_{\omega_{i}^{\ast}}(u_{0}-\psi_{i}, u_{0}-\psi_{i}) = a_{\omega_{i}^{\ast}}(u_{0}-\psi_{i}, u_{0}-\psi^{r}_{i} - \psi^{d}_{i})}\\[2mm]
{\displaystyle \quad = a_{\omega_{i}^{\ast}}(u_{0}-\psi_{i}, u_{0}) \leq \Vert u_{0}-\psi_{i}\Vert_{a,\omega_{i}^{\ast}} \Vert u_{0}\Vert_{a,\omega_{i}^{\ast}},}
\end{array}
\end{equation}
which gives \cref{eq:3-15-4}. \qquad \end{proof}

To derive an upper bound for the convergence rate of the optimal local approximation, we assume that $\omega_{i}$ and $\omega_{i}^{\ast}$ are concentric truncated cubes with side lengths $H_{i}$ and $H_{i}^{\ast}$ ($H_{i}^{\ast}>H_{i}$), respectively. Under this assumption, we have
\cm{
\begin{theorem}\label{thm:3-2}
There exist $n_{0}>0$ and $b>0$, such that for any $n>n_{0}$,
\begin{equation}
d_{n}(\omega_{i},\omega_{i}^{\ast}) \leq C_{1}e^{-1}e^{-bn^{{1}/{(d+1)}}} e^{-h(\rho)bn^{{1}/{(d+1)}}},
\end{equation}
where $C_{1}$ is again the constant given in \cref{eq:1-3}, $h(s) = 1+{s\log(s)}/{(1-s)}$, and $\rho = H_{i}/H_{i}^{\ast}$.
\end{theorem}
}

We only give a proof of \cref{thm:3-2} when $\partial \omega^{\ast}_{i} \cap \partial \Omega_{D} \neq \emptyset$. The pure Neumann boundary case can be proved in a similar way as for interior subdomains. For ease of notation, we drop again the subscript index $i$ of subdomains. 

We first introduce the closure of $H^{0}_{A,D}(\omega^{\ast})$ with respect to the $L^{2}(\omega^{\ast})$ norm and denote it by $\overline{H}^{0}_{A,D}(\omega^{\ast})$. Next we consider the following Neumann eigenvalue problem
\begin{equation}\label{eq:3-16}
a_{\omega^{\ast}}(v_{k}, \varphi) = \lambda_{k} \int_{\omega^{\ast}}v_{k}\varphi \,d{\bm x},\quad \forall \varphi \in H^{1}(\omega^{\ast}),\quad k=1,\cdots,m.
\end{equation}
Let $\Psi_{m}(\omega^{\ast})$ denote the subspace spanned by the first $m$ eigenfunctions of \cref{eq:3-16}. By the following orthogonal decomposition of $H^{1}(\omega^{\ast})$
\begin{equation}
H^{1}(\omega^{\ast}) = H_{A,D}(\omega^{\ast})\oplus H_{0D}^{1}(\omega^{\ast}),
\end{equation}
we define $W_{m}(\omega^{\ast}) = \mathcal{P}^{A}\Psi_{m}(\omega^{\ast})$, where $\mathcal{P}^{A}$ is the orthogonal projection from $H^{1}(\omega^{\ast})$ onto $H_{A,D}(\omega^{\ast})$ with respect to the inner product $a_{\omega^{\ast}}(\cdot,\,\cdot)$. Furthermore, to take the boundary conditions into account, we consider the $L^{2}$-projection of $W_{m}(\omega^{\ast})$ onto $\overline{H}^{0}_{A,D}(\omega^{\ast})$ and denote it by $\mathcal{P}_{0}W_{m}(\omega^{\ast})$, where $\mathcal{P}_{0}$ is the $L^{2}$-projection from $L^{2}(\omega^{\ast})$ onto $\overline{H}^{0}_{A,D}(\omega^{\ast})$. As for interior subdomains, we have the following approximation result.
\begin{lemma}\label{lem:3-2}
For any $u\in H_{A,D}^{0}(\omega^{\ast})$, there exists a $v_{u}\in \mathcal{P}_{0}W_{m}(\omega^{\ast})\subset \overline{H}^{0}_{A,D}(\omega^{\ast})$ such that
\begin{equation}\label{eq:3-17}
\Vert u - v_{u}\Vert_{L^{2}(\omega^{\ast})} \leq C(m)H^{\ast}\frac{\gamma_{d}^{1/d}}{\sqrt{4\pi}}\alpha^{-1/2}\Vert u \Vert_{a,\,\omega^{\ast}},
\end{equation}
where $H^{\ast}$ is the side length of the truncated cube $\omega^{\ast}$, $\gamma_{d}$ is the volume of the unit ball in $\mathbb{R}^{d}$, and $C(m)= m^{-1/d}(1+o(1))$.
\end{lemma}
\begin{proof}
As before, we consider the quantity
\begin{equation}\label{eq:3-17-0}
R= \sup_{u\in H_{A,D}^{0}(\omega^{\ast})} \inf_{v\in \mathcal{P}_{0}W_{m}(\omega^{\ast})} \frac{\Vert u-v\Vert_{L^{2}(\omega^{\ast})}}{\Vert u \Vert_{a,\,\omega^{\ast}}}.
\end{equation}
Noting that $u = \mathcal{P}_{0}u$ and $\Vert \mathcal{P}_{0}v\Vert_{L^{2}(\omega^{\ast})}\leq \Vert v\Vert_{L^{2}(\omega^{\ast})}$ for any $v\in H_{A,D}^{0}(\omega^{\ast})$, we find
\begin{equation}\label{eq:3-17-1}
\begin{array}{lll}
{\displaystyle R = \sup_{u\in H_{A,D}^{0}(\omega^{\ast})} \inf_{\phi \in W_{m}(\omega^{\ast})} \frac{\Vert  \mathcal{P}_{0}(u-\phi)\Vert_{L^{2}(\omega^{\ast})}}{\Vert u \Vert_{a,\,\omega^{\ast}}}}\\[4mm]
{\displaystyle \;\;\quad \leq \sup_{u\in H_{A,D}^{0}(\omega^{\ast})} \inf_{\phi \in W_{m}(\omega^{\ast})} \frac{\Vert  u-\phi\Vert_{L^{2}(\omega^{\ast})}}{\Vert u \Vert_{a,\,\omega^{\ast}}}. }
\end{array}
\end{equation}
Consequently, an upper bound of $R$ can be obtained by following the same lines as in the proof of \cref{lem:2-1}, from which the desired inequality \cref{eq:3-17} follows.\qquad \end{proof}

Useful properties of functions in $\overline{H}^{0}_{A,D}(\omega^{\ast})$ are stated and proved in \cref{thm:3-3} at the end of this section. They play an important role in the proof of the following lemma.
\begin{lemma}\label{lem:3-3}
Assume that $\eta\in W^{1,\infty}(\omega^{\ast})$ satisfies $\eta({\bm x})=0$ on $\partial \omega^{\ast} \cap \Omega$. For any $u\in H_{A,D}^{0}(\omega^{\ast})$, there exists a $v_{u}\in \mathcal{P}_{0}W_{m}(\omega^{\ast})\subset \overline{H}^{0}_{A,D}(\omega^{\ast})$ such that
\begin{equation}\label{eq:3-17-2}
\Vert \eta(u - v_{u})\Vert_{a,\,\omega^{\ast}} \leq C(m)H^{\ast}\frac{\gamma_{d}^{1/d}}{\sqrt{4\pi}}\Big(\frac{\beta}{\alpha}\Big)^{1/2}\Vert \nabla \eta\Vert_{L^{\infty}(\omega^{\ast})} \Vert u \Vert_{a,\,\omega^{\ast}},
\end{equation}
where $H^{\ast}$ is the side length of the truncated cube $\omega^{\ast}$, $\gamma_{d}$ is the volume of the unit ball in $\mathbb{R}^{d}$, and $C(m) = m^{-1/d}(1+o(1))$.
\end{lemma}
\begin{proof}
First we extend the Caccioppoli-type inequality to functions in $\overline{H}^{0}_{A,D}(\omega^{\ast})$, i.e.,
\begin{equation}\label{eq:3-17-3}
\Vert \eta v \Vert_{a, \omega^{\ast}} \leq \beta^{1/2} \Vert \nabla \eta \Vert_{L^{\infty}(\omega^{\ast})} \Vert v \Vert_{L^{2}(\omega^{\ast})},\quad \forall v\in \overline{H}^{0}_{A,D}(\omega^{\ast}).
\end{equation}
Let $u,\,v\in \overline{H}^{0}_{A,D}(\omega^{\ast})$. By \cref{thm:3-3}, we see that $\eta u\in H^{1}(\omega^{\ast})$ and $a_{\omega^{\ast}}(u,\,\eta^{2}v) = 0$. With the same argument as in the proof of \cref{lem:1-1}, it follows that \cref{eq:1-4-1} holds for any $u,\,v\in \overline{H}^{0}_{A,D}(\omega^{\ast})$, which gives \cref{eq:3-17-3} immediately. Applying \cref{eq:3-17-3} to $u-v_{u}$ and using \cref{lem:3-2}, we obtain \cref{eq:3-17-2}.
%Combining \cref{eq:3-17-5} and the following inequality
%\begin{equation}\label{eq:3-17-6}
%\Vert \eta v \Vert_{a, \omega^{\ast}} \leq \Big(\int_{\omega^{\ast}}(A\nabla v \cdot \nabla v) \eta^{2}\,d{\bm x}\Big)^{\frac12} +  \Big(\int_{\omega^{\ast}}(A\nabla \eta \cdot \nabla \eta) v^{2}\,d{\bm x}\Big)^{\frac12},
%\end{equation}
%we obtain \cref{eq:3-17-3}. Now  follows immediately by 
\qquad \end{proof}

Let $N\geq 1$ be an integer. Proceeding as before, we choose $\omega^{j}$, $j=1,2,\cdots,N,N+1$, to be the nested family of concentric truncated cubes with side length $H^{\ast}-\delta^{\ast} (j-1)/N$ for which $\omega=\omega^{N+1}\subset \omega^{N}\subset\cdots\subset\omega^{1} = \omega^{\ast}$, where $\delta^{\ast} = H^{\ast}-H$. Let $n=(N+1)\times m$ and define
\begin{equation}\label{eq:3-18}
\mathcal{T}(n,\omega,\omega^{\ast}) = \mathcal{P}_{0}W_{m}(\omega^{1})+ \cdots+\mathcal{P}_{0}W_{m}(\omega^{N+1}).
\end{equation}
Similar to \cref{thm:2-1}, we can prove the following convergence rate for the approximation space $\mathcal{T}(n,\omega,\omega^{\ast})$.
\begin{lemma}\label{thm:3-4}
Let $u\in H^{0}_{A,D}(\omega^{\ast})$ and $N\geq 1$ be an integer. Then there exists a $z_{u}\in \mathcal{T}(n,\omega,\omega^{\ast})$ such that
\begin{equation}\label{eq:3-19}
\Vert \chi(u -z_{u}) \Vert_{a,\omega} \leq \frac{C_{1}}{2\sqrt{2}N}\prod_{k=1}^{N-1} \big(1-\frac{k\delta^{\ast}}{NH^{\ast}}\big) \xi^{N+1} \Vert u \Vert_{a,\omega^{\ast}},
\end{equation}
where $C_{1}$ is the positive constant defined in \cref{eq:1-3} and $\xi$ is given by
\begin{equation}\label{eq:3-20}
\xi =\xi(N,m)=N\frac{\gamma_{d}^{1/d}}{\sqrt{\pi}}\Big(\frac{\beta}{\alpha}\Big)^{1/2} \frac{H^{\ast}}{\delta^{\ast}}C(m).
\end{equation}
\end{lemma}

\cm{
The proof of \cref{thm:3-2} now follows as before for interior subdomains by recalling the definition of $n_{0}$ and $b$ in \cref{eq:1-5-5-1} to prove that if $n\geq n_{0}$, then
\begin{equation}\label{eq:3-21}
\begin{array}{lll}
{\displaystyle d_{n}(\omega,\omega^{\ast}) \leq \sup_{u\in H^{0}_{A,D}(\omega^{\ast})} \inf_{v\in \mathcal{T}(n,\omega,\omega^{\ast})}\frac {\Vert \chi(u-v)\Vert_{a,\omega}}{\Vert u \Vert_{a,\omega^{\ast}}} }\\[4mm]
{\displaystyle \;\; \qquad \qquad \leq C_{1}e^{-1}e^{-bn^{{1}/{(d+1)}}} e^{-h(\rho)bn^{{1}/{(d+1)}}},}
\end{array}
\end{equation}
where $C_{1}$ is the constant given in \cref{eq:1-3}, $h(s) = 1+{s\log(s)}/{(1-s)}$, and $\rho = H/H^{\ast}$.}

We end this section by stating and proving the following lemma used in the proof of \cref{lem:3-3}.
\begin{lemma}\label{thm:3-3}
Let $u_{\infty}\in \overline{H}^{0}_{A,D}(\omega^{\ast})$. For any open set $\mathcal{O}\subset \omega^{\ast}$ with $dist(\partial \mathcal{O}, \\ \partial \omega^{\ast} \cap \Omega) > 0$, $u_{\infty}\in H^{1}(\mathcal{O})$ and
\begin{equation}\label{eq:3-21-2}
\begin{array}{lll}
{\displaystyle \qquad \quad u_{\infty} = 0\;\;{\rm on}\;\, \partial \mathcal{O} \cap (\partial \omega^{\ast} \cap \partial \Omega_{D}),\quad {\rm if}\;\; \partial \mathcal{O} \cap (\partial \omega^{\ast} \cap \partial \Omega_{D}) \neq \emptyset.}
\end{array}
\end{equation}
In addition, for any $\eta\in {W}^{1,\infty}(\omega^{\ast})$ satisfying $\eta({\bm x}) = 0$ on $\partial \omega^{\ast}\cap \Omega$, $\eta u_{\infty}\in H^{1}(\omega^{\ast})$ and
\begin{equation}\label{eq:3-21-1}
a_{\omega^{\ast}}(u_{\infty},\,\eta^{2}v) = 0,\quad \forall v\in \overline{H}^{0}_{A,D}(\omega^{\ast}).
\end{equation}
\end{lemma}
\begin{proof}
By definition, there exists a sequence $\{u_{m}\}_{m=1}^{\infty}\subset {H}^{0}_{A,D}(\omega^{\ast})$ such that $u_{m}\rightarrow u_{\infty}$ in $L^{2}(\omega^{\ast})$ as $m\rightarrow\infty$. Assume that $\mathcal{O}$ is an open subset of $\omega^{\ast}$ with $dist(\partial \mathcal{O}, \;\partial \omega^{\ast} \cap \Omega) > 0$. We introduce a cut-off function $\eta\in W^{1,\infty}(\omega^{\ast})$ satisfying
\begin{equation}
0\leq\eta\leq 1;\quad\; \eta({\bm x}) = 1\quad {\rm in}\;\, \mathcal{O}; \quad\; \eta({\bm x}) = 0 \quad {\rm on}\;\, \partial \omega^{\ast} \cap \Omega.
\end{equation}
For $m$, $l\in \mathbb{N}^{+}$, since $u_{m}-u_{l}$ is $A$-harmonic, applying \cref{eq:3-5} gives 
%and $\eta^{2}(u_{m} - u_{l})\in H^{1}_{0D}(\omega^{\ast})$, we find 
%\begin{equation}\label{eq:3-21-3}
%a_{\omega^{\ast}}(u_{m} - u_{l},\eta^{2}(u_{m} - u_{l})) = 0.
%\end{equation}
%Using a similar argument as in the proof of \cref{lem:1-1}, we get
\begin{equation}\label{eq:3-21-4}
\Vert \eta(u_{m} - u_{l})\Vert_{a,\,\omega^{\ast}} \leq \beta^{\frac12} \Vert \nabla \eta \Vert_{L^{\infty}(\omega^{\ast})} \Vert u_{m} - u_{l}\Vert_{L^{2}(\omega^{\ast})},
\end{equation}
and thus
\begin{equation}
\Vert u_{m} - u_{l} \Vert_{a,\,\mathcal{O}}\leq \beta^{\frac12} \Vert \nabla \eta \Vert_{L^{\infty}(\omega^{\ast})} \Vert u_{m} - u_{l}\Vert_{L^{2}(\omega^{\ast})},
\end{equation}
which implies that $\{u_{m}\}_{m=1}^{\infty}$ is a Cauchy sequence in $H^{1}(\mathcal{O})$. Hence, we have $u_{m}\rightarrow u_{\infty}$ in $H^{1}(\mathcal{O})$ and $u_{\infty} \in H^{1}(\mathcal{O})$. Let $\gamma$ be the trace operator. We have
\begin{equation}
\begin{array}{lll}
{\displaystyle \Vert \gamma u_{\infty}\Vert_{H^{1/2}(\partial \mathcal{O} \cap (\partial \omega^{\ast} \cap \partial \Omega_{D}))} = \Vert  \gamma (u_{\infty} - u_{m}) \Vert_{H^{1/2}(\partial \mathcal{O} \cap (\partial \omega^{\ast} \cap \partial \Omega_{D}))}}\\[2mm]
{\displaystyle\quad \leq\, C \Vert u_{\infty} - u_{m}\Vert_{H^{1}(\mathcal{O})} \rightarrow 0,\quad {\rm as} \;\,m\rightarrow\infty,}
\end{array}
\end{equation}
which yields that $\gamma u_{\infty} = 0$ on $\partial \mathcal{O} \cap (\partial \omega \cap \partial \Omega_{D})$ and thus \cref{eq:3-21-2} is proved. Note that \cref{eq:3-21-4} hold for any $\eta\in {W}^{1,\infty}(\omega^{\ast})$ with $\eta({\bm x}) = 0$ on $\partial \omega^{\ast}\cap \Omega$. Hence, \cref{eq:3-21-4} implies that $\{\eta u_{m}\}_{m=1}^{\infty}$ is a Cauchy sequence in $H^{1}(\omega^{\ast})$ and we see that $\eta u_{\infty}\in H^{1}(\omega^{\ast})$. Now, it remains to prove \cref{eq:3-21-1}. Let $v\in \overline{H}^{0}_{A,D}(\omega^{\ast})$. We see that $\eta^{2} v\in H^{1}_{0D}(\omega^{\ast})$ and consequently,
\begin{equation}\label{eq:3-21-5}
a_{\omega^{\ast}}(u_{m}, \eta^{2} v) =0,\quad \forall m\in \mathbb{N}^{+},\;v\in \overline{H}^{0}_{A,D}(\omega^{\ast}).
\end{equation}
Note that \cref{eq:3-21-1} doesn't follow immediately from \cref{eq:3-21-5} since in general we don't have $\Vert u_{m} - u_{\infty}\Vert_{a,\omega^{\ast}}\rightarrow 0$. To prove \cref{eq:3-21-1}, we first show that
\begin{equation}\label{eq:3-21-6}
\Vert \eta A^{1/2}\nabla (u_{m}-u_{\infty})\Vert_{L^{2}(\omega^{\ast})}\rightarrow 0,\quad {\rm as} \;\,m\rightarrow\infty.
\end{equation}
By triangle inequalities, we have
\begin{equation}\label{eq:3-21-7}
\begin{array}{lll}
{\displaystyle \Vert \eta A^{1/2}\nabla (u_{m}-u_{\infty})\Vert_{L^{2}(\omega^{\ast})} \leq \Vert  (u_{m}-u_{\infty}) A^{1/2}\nabla\eta \Vert_{L^{2}(\omega^{\ast})} }\\[2mm]
{\displaystyle \qquad \qquad +\, \Vert  A^{1/2}\nabla (\eta u_{m}-\eta u_{\infty})\Vert_{L^{2}(\omega^{\ast})}. }
\end{array}
\end{equation}
Now \cref{eq:3-21-6} follows from \cref{eq:3-21-7} and the strong convergence of $\{u_{m}\}_{m=1}^{\infty}$ and $\{\eta u_{m}\}_{m=1}^{\infty}$ in $L^{2}(\omega^{\ast})$ and $H^{1}(\omega^{\ast})$, respectively. A similar argument yields that $\eta A^{1/2} \nabla v \in L^2{(\omega^{\ast})}$ for any $v\in \overline{H}^{0}_{A,D}(\omega^{\ast})$. By \cref{eq:3-21-5}, we find that for any $v\in \overline{H}^{0}_{A,D}(\omega^{\ast})$,
\begin{equation}\label{eq:3-21-8}
\begin{array}{lll}
{\displaystyle a_{\omega^{\ast}}(u_{\infty}, \eta^{2} v) = a_{\omega^{\ast}}(u_{\infty}-u_{m}, \eta^{2} v) }\\[3mm]
{\displaystyle = \int_{\omega^{\ast}}A\nabla(u_{\infty}-u_{m})\cdot (2\eta v\nabla\eta  + \eta^{2} \nabla v)d{\bm x}}\\[4mm] 
{\displaystyle \leq \Vert \eta A^{1/2}\nabla (u_{\infty}-u_{m})\Vert_{L^{2}(\omega^{\ast})} \big(2\Vert vA^{1/2}\nabla \eta\Vert_{L^{2}(\omega^{\ast})} + \Vert \eta A^{1/2} \nabla v\Vert_{L^{2}(\omega^{\ast})}\big),}
\end{array}
\end{equation}
which yields \cref{eq:3-21-1} by applying \cref{eq:3-21-7}. \qquad \end{proof}
\begin{rem}
In general, it is difficult to prove $a_{\omega^{\ast}}(u_{\infty},v) =0$ for all $v\in H^{1}_{0D}(\omega^{\ast})$.
\end{rem}

\section{Numerical implementation}\label{sec-4}
In this section, we discuss the numerical implementation of the multiscale GFEM in detail. Instead of using the partition of unity functions, in the discrete setting we use the local partition of unity operators introduced in \cite{spillane2014abstract} to generate and glue together the local approximation spaces.
Special focus is put on the efficient generation of the discrete $A$-harmonic spaces.

Assume that $\Omega$ is a Lipschitz polygonal (polyhedral) domain. Let $\mathcal{T}_{h}=\{K\}$ be a regular partition of $\Omega$ into triangles (quadrilaterals) in $\mathbb{R}^{2}$ or tetrahedrons (hexahedrons) in $\mathbb{R}^{3}$, where $h=\max_{K\in \mathcal{T}_{h}}\{diam (K)\}$. The mesh-size $h$ is assumed to be small enough to resolve all fine-scale details of the coefficient $A({\bm x})$. Let $V_{h}$ be a conforming finite element space of $H^{1}(\Omega)$ with a basis of piecewise linear functions $\{\varphi_{k}\}_{k=1}^{n}$, where $n$ is the dimension of $V_{h}$. We first partition $\Omega$ into a set of non-overlapping subdomains resolved by $\mathcal{T}_{h}$ and then extend each subdomain by adding several layers of mesh elements to create an overlapping decomposition $\{\omega_{i}\}_{i=1}^{M}$ of $\Omega$.

For each $i=1,\cdots,M$, we define the following finite element spaces on $\omega_{i}$
\begin{equation}\label{eq:4-1}
\begin{array}{lll}
{\displaystyle  {V}_{h}(\omega_{i}) = \big\{v|_{\omega_{i}}\;:\; v\in V_{h}\big\},   }\\[2mm]
{\displaystyle V_{h,0}(\omega_i)= \big\{v|_{\omega_{i}}\;:\; v\in V_{h},\; {\rm supp}(v)\subset \overline{\omega_{i}}\big\} ,}
\end{array}
\end{equation}
as well as the set of internal degrees of freedom in $\omega_{i}$
\begin{equation}\label{eq:4-2}
\tmop{dof}(\omega_i):=\big\{k:\;1\leq k\leq n \;\;{\rm and}\;\;{\rm supp}(\varphi_{k})\subset \overline{\omega_{i}}\big\}.
\end{equation}
Moreover, we denote by $R_{i}^{T}$ the zero extension operator, which extends a function $v\in V_{h,0}(\omega_{i})$ by zero to $V_{h}$. Next we introduce the local partition of unity operators associated with the overlapping partition $\{\omega_{i}\}_{i=1}^{M}$, which are the discrete analog of the partition of unity functions introduced in \cref{eq:1-3}.

\begin{definition}[Partition of unity operators]
For any degree of freedom $k$ ($1\leq k\leq n$), let $\mu_{k}$ denote the number of subdomains for which $k$ is an internal degree of freedom, i.e.,
\begin{equation}\label{eq:4-3}
\mu_{k}:=\#\big\{j\;:\;1\leq j\leq M,\;\;k\in \tmop{dof}(\omega_j)\big\}.
\end{equation}
For each $j=1,\cdots, M$, the local partition of unity operator $\Xi_{j}:V_{h}(\omega_{j})\rightarrow V_{h,0}(\omega_j)$ is defined by
\begin{equation}\label{eq:4-4}
\displaystyle \Xi_{j}(v) :=\sum_{k\in \tmop{dof}(\omega_j)}\frac{1}{\mu_k} v_{k}\varphi_{k}|_{\omega_j} \quad \forall\, v= \sum_{k\geq 1}v_{k}\varphi_{k} \in V_{h}(\omega_{j}).
\end{equation}
\end{definition}
It can be proved \cite{spillane2014abstract} that the operators $\Xi_{j}$ satisfy
\begin{equation}\label{eq:4-5}
\displaystyle \sum_{j=1}^{M} R^{T}_{j}\Xi_{j}(v|_{\omega_j}) = v,\quad \forall v\in V_{h},
\end{equation}
and
\begin{equation}\label{eq:4-6}
\Xi_{j}(v)|_{\omega_{j}\backslash\omega_{j}^{\circ}} = v|_{\omega_{j}\backslash\omega_{j}^{\circ}},\quad \forall v\in V_{h}(\omega_{j}), \;\; j=1,\cdots,M,
\end{equation}
where $\omega_{j}^{\circ} =\{x\in \omega_{j}:\;\exists\, j^{\prime}\neq j\;\;{\rm such \;that}\;\; x\in\omega_{j^{\prime}}\}$ denotes the overlapping zone.

To proceed, we extend each subdomain $\omega_{i}$ by adding several layers of mesh elements to create a larger domain $\omega_{i}^{\ast}$ on which the local particular function and the optimal local approximation space are built. The subdomains $\{\omega^{\ast}_{i}\}$ are usually referred to as the oversampling domains. For each $i=1,\cdots,M$, we define the space of restrictions of functions in $V_{h}$ to $\omega_{i}^{\ast}$ in which homogeneous Dirichlet boundary conditions on $\partial \Omega_{D}$ are incorporated as follows.
\begin{equation}\label{eq:4-7}
\begin{array}{lll}
{\displaystyle  {V}_{h}(\omega^{\ast}_{i}) = \big\{v|_{\omega^{\ast}_{i}}\;:\; v\in V_{h}\big\}, }\\[2mm]
{\displaystyle {V}_{hD}(\omega^{\ast}_{i}) = \big\{v\in V_{h}(\omega_{i}^{\ast}):\; v = 0 \;\;{\rm on}\;\, \partial \omega^{\ast}_{i} \cap \partial \Omega_{D}\big\},  }\\[2mm]
{\displaystyle V_{h,0}(\omega^{\ast}_i)= \big\{v\in {V}_{hD}(\omega^{\ast}_{i}):\; {\rm supp}(v)\subset \overline{\omega^{\ast}_{i}}\big\}, }\\[2mm]
{\displaystyle W_{h}(\omega^{\ast}_i)= \big\{u\in V_{hD}(\omega^{\ast}_i)\;:\; a_{\omega^{\ast}_{i}}(u,v) = 0,\;\, \forall v\in  V_{h,0}(\omega^{\ast}_i)\big\}. }
\end{array}
\end{equation}
We see that functions in $V_{h,0}(\omega^{\ast}_i)$ vanish on $(\partial \omega^{\ast}_{i} \cap \partial \Omega_{D})\cup (\partial \omega^{\ast}_{i} \cap \Omega)$, while $W_{h}(\omega^{\ast}_i)$ denote the discrete $A$-harmonic spaces.

\begin{rem}
For each $i=1,\cdots,M$, $W_{h}(\omega^{\ast}_i)$ is spanned by the $A$-harmonic extensions of the hat functions corresponding to the nodes on the boundary and thus the dimension of $W_{h}(\omega^{\ast}_i)$ is equal to the number of degrees of freedom on $\partial \omega^{\ast}_{i}$. 
\end{rem}

On each subdomain $\omega_{i}$, the local particular function is defined as $u^{p}_{h,i} = (\psi_{h,i}^{r} + \psi_{h,i}^{d})|_{\omega_{i}}$, where $\psi_{h,i}^{r} \in V_{h,0}(\omega_{i}^{\ast})$ satisfies 
\begin{equation}\label{eq:4-9}
a_{\omega_{i}^{\ast}}(\psi^{r}_{h,i}, v) = F_{\omega_{i}^{\ast}}(v),\quad \forall v\in V_{h,0}(\omega^{\ast}_i)
\end{equation}
and $\psi_{h,i}^{d} \in V_{h}(\omega_{i}^{\ast})$ satisfies $\psi_{h,i}^{d}({\bm x})=q({\bm x})$ on $\partial \omega_{i}^{\ast}\cap \partial \Omega_D$ and 
\begin{equation}\label{eq:4-10}
a_{\omega_{i}^{\ast}}(\psi^{d}_{h,i}, v) =0,\quad \forall v\in V_{hD}(\omega^{\ast}_i).
\end{equation}
Note that $\psi_{h,i}^{d}$ vanishes if $\partial \omega_{i}^{\ast} \cap \partial \Omega_{D}=\emptyset$ and $F_{\omega_{i}^{\ast}}(v)=(f,\,v)_{L^{2}(\omega_{i}^{\ast})}$ if $\partial \omega_{i}^{\ast} \cap \partial \Omega_{N}=\emptyset$.

On each subdomain $\omega_{i}$, the local approximation space $S_{h,n_i}(\omega_{i})$ is defined as 
\begin{equation}\label{eq:4-11}
S_{h,n_i}(\omega_{i}) = {\rm span}\big\{\phi_{h,1}|_{\omega_{i}},\cdots,\phi_{h,n_{i}}|_{\omega_{i}}\big\},
\end{equation}
where $\{\phi_{h,j}\}_{j=1}^{n_i}$ are the eigenfunctions corresponding to the $n_{i}$ smallest eigenvalues of the following eigenvalue problem:
\begin{equation}\label{eq:4-12}
a_{\omega^{\ast}_{i}}(\phi,v) = \lambda \,a_{\omega_{i}}(\Xi_{i}(\phi|_{\omega_{i}}), \,\Xi_{i}(v|_{\omega_{i}})),\quad \forall \,v\in {W}_{h}(\omega^{\ast}_i).
\end{equation}
The global particular function and test space for the GFEM are then defined by
\begin{equation}\label{eq:4-13}
\begin{array}{lll}
{\displaystyle u_{h}^{p} :=\sum_{i=1}^{M}R_{i}^{T}\Xi_{i}(u^{p}_{h,i})\;\;\;{\rm and}\;\; \; S_{h}(\Omega):= \Big\{ \sum_{i=1}^{M} R^{T}_{i}\Xi_{i}(v_{i}):\; v_{i}\in S_{h,n_i}(\omega_{i}) \Big\}. }
\end{array}
\end{equation}
The final step of the MS-GFEM algorithm is to solve the problem \cref{eq:1-2-3} on the test space $S_{h}(\Omega)$: Find $u_{h}^{s} \in S_{h}(\Omega)$ such that
\begin{equation}\label{eq:4-14}
a(u_{h}^{s},\,v) = F(v)-a(u^{p}_{h}, v),\quad \forall v\in S_{h}(\Omega),
\end{equation}
and form the approximate solution by $u_{h}^{G}= u_{h}^{p} + u_{h}^{s}$.

Most of the computational work of the original MS-GFEM in \cite{babuska2011optimal,babuvska2020multiscale} lies in the generation of the discrete $A$-harmonic spaces, which in general requires the solution of a large number of local boundary value problems. To get $n_{i}$ eigenfunctions for constructing a local approximation space on $\omega_{i}$, it was suggested in \cite{babuvska2020multiscale} to use an approximation of the discrete $A$-harmonic space $W_{h}(\omega^{\ast}_{i})$ spanned by the $A$-harmonic extension of $\tilde{n}_{i}>n_{i}$ suitably chosen FE functions on $\partial \omega_{i}^{\ast}$. In this paper, we dramatically reduce this cost by solving the Steklov eigenvalue problem associated with the Dirichlet-to-Neumann (DtN) operator on $\partial \omega_{i}^{\ast}$ and use those eigenfunctions to generate the discrete $A$-harmonic spaces. It is worth noting that a similar eigenvalue problem was used to build coarse spaces for two-level additive Schwarz methods \cite{dolean2012analysis}.

We introduce the eigenvalue problems
\begin{equation}\label{eq:4-15}
\left\{
\begin{array}{lll}
{\displaystyle -{\rm div}(A\nabla u) = 0,\;\;\quad {\rm in}\;\, \omega^{\ast}_{j} }\\[2mm]
{\displaystyle {\bm n} \cdot A\nabla u=\lambda u, \quad \quad \;\;{\rm on}\;\,\partial \omega_{j}^{\ast}}
\end{array}
\right.
\end{equation}
and 
\begin{equation}\label{eq:4-16}
\left\{
\begin{array}{lll}
{\displaystyle -{\rm div}(A\nabla u) = 0,\;\;\,\quad {\rm in}\;\, \omega^{\ast}_{j} }\\[2mm]
{\displaystyle {\bm n} \cdot A\nabla u=\lambda u, \quad \quad \;\;{\rm on}\;\,\partial \omega_{j}^{\ast}\cap \Omega}\\[2mm]
{\displaystyle {\bm n} \cdot A\nabla u=0, \quad \qquad \;{\rm on}\;\,\partial \omega_{j}^{\ast}\cap \partial \Omega_{N}}\\[2mm]
{\displaystyle u = 0,\qquad \qquad \qquad \,\;{\rm on}\;\,\partial \omega_{j}^{\ast}\cap \partial \Omega_{D}}
\end{array}
\right.
\end{equation}
for subdomains that lie in the interior of $\Omega$ and those that intersect the boundary of $\Omega$, respectively. In discrete variational form, the eigenvalue problems \cref{eq:4-15,eq:4-16} can be written in a unified way.

\begin{definition}[{Steklov Eigenproblem}]\label{def:1-3}
For each $j=1,\cdots,M$, we define the following eigenvalue problem
\begin{equation}\label{eq:4-17}
a_{\omega^{\ast}_{j}}(\phi,v) = \lambda \,b_{\omega^{\ast}_{j}}(\phi,v),\quad \forall v\in {V}_{hD}(\omega^{\ast}_j),
\end{equation}
where $\displaystyle b_{\omega^{\ast}_{j}}(\phi,v) = \int_{\partial \omega^{\ast}_j\cap \Omega} \phi v\,d{\bm s}$ for all $\phi$, $v\in {V}_{hD}(\omega^{\ast}_j)$.
\end{definition}

The following lemma gives a characterization of the spaces $W_{h}(\omega^{\ast}_j)$ and $V_{h,0}(\omega^{\ast}_j)$ via the eigenfunctions of the eigenvalue problem \cref{eq:4-17}.
\begin{lemma}\label{lem:4-1}
For each $j$, consider the eigenvalue problem \cref{eq:4-17} in \cref{def:1-3}

\vspace{2mm}
$(i)$ There are $N_j = {\rm dim}(W_{h}(\omega^{\ast}_j))$ finite eigenvalues $0\leq \lambda_{1}^{j}\leq  \cdots \leq \lambda_{N_j}^{j}<\infty$ (counted according to multiplicity) with corresponding eigenfunctions $\{\phi_{k}^{j}\}_{k=1}^{N_j}$, which can be normalized to form an orthonormal basis of $W_{h}(\omega^{\ast}_j)$ with respect to $b_{\omega^{\ast}_j}(\cdot,\cdot)$.

\vspace{2mm}
$(ii)$ There are $K_j =  {\rm dim}(V_{h,0}(\omega^{\ast}_j))$ infinite eigenvalues $\lambda_{1}^{j}=\cdots=\lambda_{K_j}^{j}=\infty$ with associated eigenfunctions $\{\varphi_{i}^{j}\}_{i=1}^{K_j}$ forming a basis of $V_{h,0}(\omega^{\ast}_j)$. Here each $\varphi_{i}^{j}$ satisfies
\begin{equation}\label{eq:4-18}
b_{\omega^{\ast}_{j}}(\varphi_{i}^{j},v) = 0,\quad  \forall v\in V_{h,0}(\omega^{\ast}_j)\quad {\rm and}\quad a_{\omega^{\ast}_{j}}(\varphi_{i}^{j}, \phi_{k}^{j})=0,\quad \forall \,k=1,\cdots, N_{j}.
\end{equation}
\end{lemma}
\begin{proof}
First we observe that ${V}_{hD}(\omega^{\ast}_{j}) = V_{h,0}(\omega^{\ast}_j) \oplus W_{h}(\omega^{\ast}_j)$. Since $a_{\omega^{\ast}_j}(u,\,v) = b_{\omega^{\ast}_j}(u,\,v) = 0$ for all $u\in V_{h,0}(\omega^{\ast}_j) $ and $v\in W_{h}(\omega^{\ast}_j)$, the eigenproblem \cref{eq:4-17} can be decoupled into two eigenproblems defined on $V_{h,0}(\omega^{\ast}_j)$ and $W_{h}(\omega^{\ast}_j)$ separately. 

Next we show that $b_{\omega^{\ast}_j}(\cdot,\cdot)$ is positive definite on $W_{h}(\omega^{\ast}_j)\times W_{h}(\omega^{\ast}_j)$. Let $v\in W_{h}(\omega^{\ast}_j)$ such that $b_{\omega^{\ast}_j}(v,v) = 0$. By definition, we know that
\begin{equation}
\int_{\partial \omega_{j}^{\ast}\cap \Omega} |w|^{2}d{\bm s} = 0.
\end{equation}
\vspace{-3mm}

Therefore, $v \in V_{h,0}(\omega^{\ast}_j)$. Since $V_{h,0}(\omega^{\ast}_j) \cap W_{h}(\omega^{\ast}_j) = \{ 0\}$, we see that $v=0$ and thus $b_{\omega^{\ast}_j}(\cdot,\cdot)$ is positive definite on $W_{h}(\omega^{\ast}_j)\times W_{h}(\omega^{\ast}_j)$. Now we consider the restriction of \cref{eq:4-17} to $W_{h}(\omega^{\ast}_j)$. Since the bilinear forms $a_{\omega^{\ast}_{j}}(\cdot,\cdot)$ and $b_{\omega^{\ast}_{j}}(\cdot,\cdot)$ are positive semi-definite and positive definite on $W_{h}(\omega^{\ast}_j)\times W_{h}(\omega^{\ast}_j)$, respectively, the generalized eigenproblem \cref{eq:4-17} can be reduced to a standard eigenvalue problem and the assertion $(i)$ follows from standard spectral theory.

To prove $(ii)$, we consider the restriction of \cref{eq:4-17} to $V_{h,0}(\omega^{\ast}_j)$. Note that $a_{\omega^{\ast}_{j}}(\cdot,\cdot)$ is coercive on 
$V_{h,0}(\omega^{\ast}_j)$ and $b_{\omega^{\ast}_{j}}(u,v) = 0$ for all $u,\,v\in V_{h,0}(\omega^{\ast}_j)$. Therefore, all functions in $V_{h,0}(\omega^{\ast}_j) \backslash\{0\}$ are eigenfunctions associated with the eigenvalue $+\infty$ in the sense of \cref{eq:4-18}. In particular, \cref{eq:4-18} holds for any basis of $V_{h,0}(\omega^{\ast}_j)$.\qquad \end{proof}

\Cref{lem:4-1} indicates that the eigenfunctions corresponding to the finite eigenvalues of \cref{eq:4-17} form a basis of $W_{h}(\omega^{\ast}_j)$. Therefore, we can generate the discrete $A$-harmonic spaces by solving the eigenvalue problem \cref{eq:4-17}. In fact, it is not necessary to use all the eigenfunctions. The discrete $A$-harmonic spaces constructed by a handful of eigenfunctions can yield good numerical results in practice. To see this, consider the eigenproblem \cref{eq:4-17} restricted to $W_{h}(\omega^{\ast}_j)$, i.e.,
\begin{equation}\label{eq:4-19}
a_{\omega^{\ast}_{j}}(\phi,v) = \lambda \,b_{\omega^{\ast}_{j}}(\phi,v),\quad \forall v\in W_{h}(\omega^{\ast}_j)
\end{equation}
and denote by $V^{j}_{n} = {\rm span}\{ \phi_{k}^{j}\}_{k=1}^{n}$ the subspace spanned by the eigenfunctions corresponding to the $n$ smallest eigenvalues $\big(\lambda^{j}_{k}\big)_{k=1}^{n}$ of \cref{eq:4-19}. Using the characterization of the Kolmogorov $n$-width of the (compact) trace operator $T:W_{h}(\omega^{\ast}_j) \rightarrow L^{2}(\partial \omega_{j}^{\ast}\cap\Omega)$ and a similar argument as in \cref{sec-3}, it follows that for all $u\in W_{h}(\omega^{\ast}_j)$,
\begin{equation}\label{eq:4-20}
\inf_{v\in V^{j}_{n}}\Vert u -v\Vert^{2}_{b, \, \omega^{\ast}_{j}}\leq \frac{1}{\lambda_{n+1}^{j}}\Vert u \Vert^{2}_{a, \, \omega^{\ast}_{j}}.
\end{equation}
Since all norms on a finite-dimensional space are equivalent, there exists a constant $C$ independent of $n$, but possibly depending on $h$ such that
\begin{equation}\label{eq:4-21}
\inf_{v\in V^{j}_{n}}\Vert u -v\Vert^{2}_{a, \, \omega^{\ast}_{j}}\leq \frac{C}{\lambda_{n+1}^{j}}\Vert u \Vert^{2}_{a, \, \omega^{\ast}_{j}}.
\end{equation}
Therefore, the span of the first $n$ eigenfunctions of \cref{eq:4-19} (also the first $n$ eigenfunctions of \cref{eq:4-17} by \cref {lem:4-1}) can be used as an approximation of $W_{h}(\omega^{\ast}_j)$ and the error is controlled by $1/\lambda_{n+1}^{j}$. Denoting by $\widetilde{\lambda}_{k}^{j}$ the $k$-th eigenvalue of the continuous Steklov eigenproblems \cref{eq:4-15} or \cref{eq:4-16} and using the minimax principle and eigenvalue asymptotics for Stekolv eigenproblems \cite[Chapter VI]{Courant1989}, we get
\begin{equation}\label{eq:4-22}
1/\lambda_{n+1}^{j} \leq 1/\widetilde{\lambda}_{n+1}^{j}\rightarrow 0 \quad {\rm as}\quad n\rightarrow \infty.
\end{equation}
\cm{The combination of \cref{eq:4-21,eq:4-22} provides some intuition as to why the span of a few Steklov eigenfunctions can be used as an approximation of the discrete $A$-harmonic space. A complete and rigorous justification of the use of Steklov eigenproblems is left for future work.}

%Our rule of thumb is that if one wants to construct a local space with $n_{i}$ spectral basis functions, building the corresponding discrete $A$-harmonic subspace with $3n_{i}\sim 4n_{i}$ eigenfunctions of \cref{eq:4-17} are usually capable to produce good numerical results.

We conclude this section by outlining the main steps of the MS-GFEM algorithm.
\begin{enumerate}
\item Create a fine FE mesh over the entire domain $\Omega$ and define an overlapping decomposition $\{\omega_{i}\}_{i=1}^{M}$ of $\Omega$ resolved by the mesh, which is then extended to a decomposition into larger domains $\{\omega^{\ast}_{i}\}_{i=1}^{M}$ with $\omega_{i}\subset\omega_{i}^{\ast}$.

\vspace{1mm}
\item For $i=1,\cdots,M$, 
\begin{itemize}
\item Solve \cref{eq:4-9} on the oversampling domain $\omega_{i}^{\ast}$ to get the local particular function $\psi_{h,i}^{r} \in V_{h,0}(\omega_{i}^{\ast})$, as well as \cref{eq:4-10} if $\partial \omega_{i}^{\ast} \cap \partial \Omega_{D} \neq \emptyset$.

\vspace{1mm}
\item Solve the eigenproblem \cref{eq:4-17} on the oversampling domain $\omega_{i}^{\ast}$ to construct a subspace of the discrete $A$-harmonic space $W_{h}(\omega_{i}^{\ast})$.

\vspace{1mm}
\item Solve the eigenproblem \cref{eq:4-12} over the constructed discrete $A$-harmonic subspace to build the local approximation space on $\omega_{i}$.
\end{itemize}

\vspace{1mm}
\item Build the global particular function and the global test space via \cref{eq:4-13} and then solve \cref{eq:4-14} to get the approximate solution.
\end{enumerate}
\vspace{2mm}
It is important to note that in step 2, which contains by far the bulk of the computational work of the algorithm, all steps can be performed fully in parallel without any communication. This is one of the main merits of the MS-GFEM.

\section{Numerical examples}\label{sec-5}
\begin{figure}[!htbp]
\centering
\includegraphics[scale=0.29]{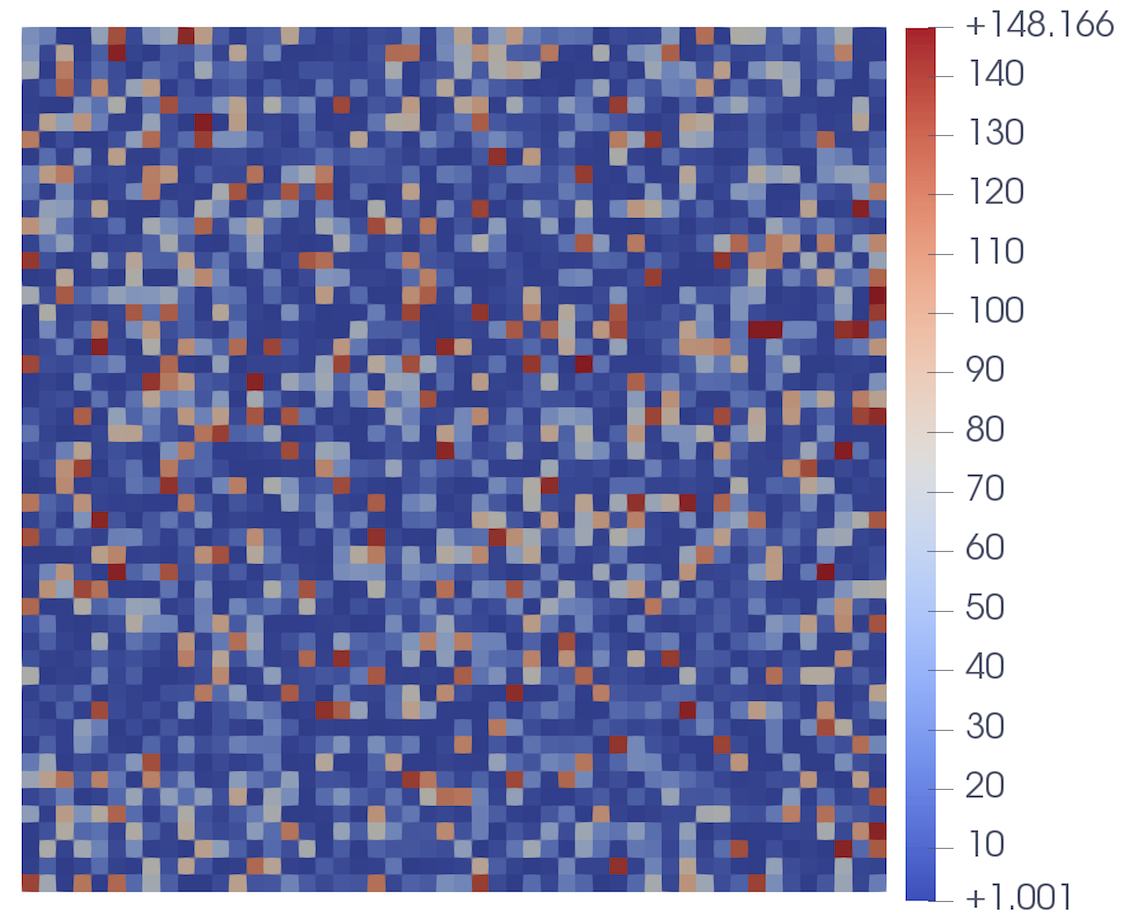}~
\includegraphics[scale=0.32]{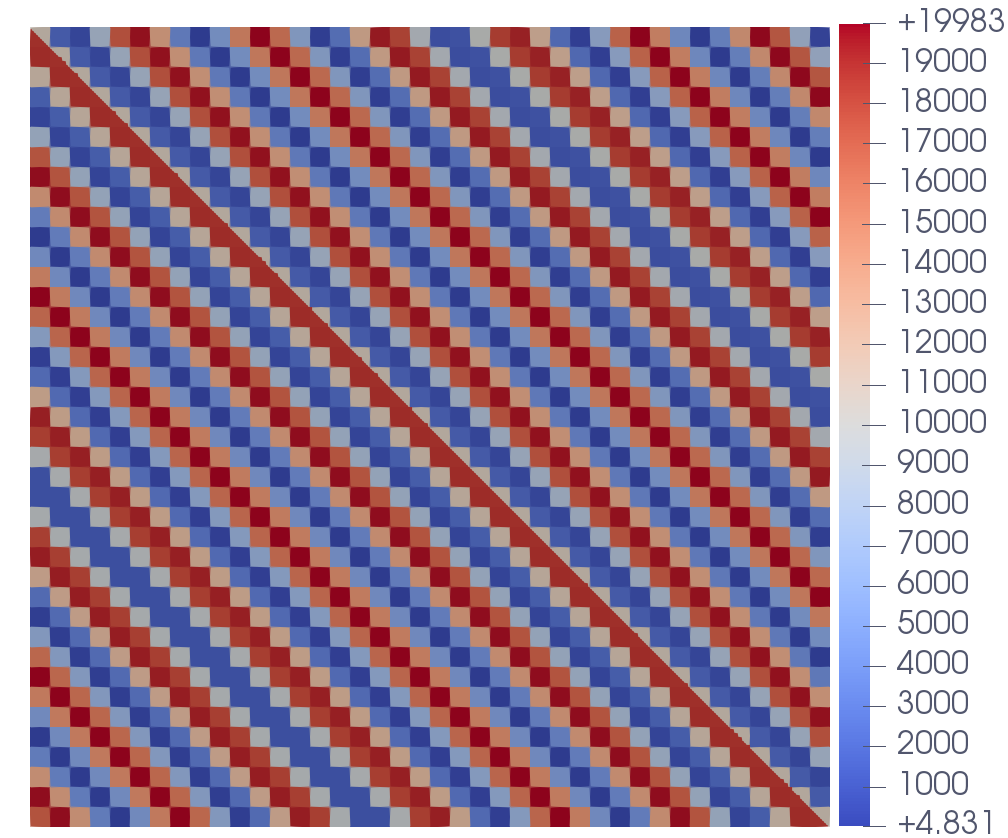}
\caption{\cm{The coefficient $A({\bm x})$ for the random field example (left) and the high contrast example (right).}}\label{fig5-1}
\end{figure}
In this section, we perform numerical experiments to support our theoretical analysis and demonstrate the effectiveness of our method. We consider the following problem on the domain $\Omega=[0,1]^{2}$:
\begin{equation}\label{eq:5-1}
\left\{
\begin{array}{lll}
{\displaystyle -{\rm div}(A({\bm x})\nabla u({\bm x})) = f({\bm x}),\;\quad {\rm in}\,\;\Omega }\\[2mm]
{\displaystyle {\bm n} \cdot A({\bm x})\nabla u({\bm x})=-1, \qquad \quad \;\;{\rm on}\;\,\partial \Omega_{N}}\\[2mm]
{\displaystyle u({\bm x}) = 1, \;\;\;\qquad \qquad \qquad \qquad {\rm on}\;\,\partial \Omega_{D},}
\end{array}
\right.
\end{equation}
where $\partial \Omega_{N} = \big\{(x_{1}, x_{2})\in \Omega: x_{2}=0 \,\;{\rm or}\;\, x_{2}=1\big\}$, $\partial \Omega_{D} = \big\{(x_{1}, x_{2})\in \Omega: x_{1}=0 \,\;{\rm or}\;\, x_{1}=1\big\}$. For the coefficient $A({\bm x})$ and the source term $f({\bm x})$, we consider the following two examples:

\begin{itemize}
\item \textbf{A random field example}. The coefficient $A({\bm x})$ is chosen to be a scalar piecewise constant function varying at a scale of $1/50$ with values taken from a random variable as illustrated in \cref{fig5-1} (left). The source term $f({\bm x})$ is given by
\begin{equation*}
f({\bm x}) = 10^{3}\times\exp\big(-10(x_{1}-0.35)^{2}-10(x_{2}-0.55)^{2}\big).
\end{equation*}

\item \cm{\textbf{A high contrast example}. The coefficient $A({\bm x})$ is chosen to be 
\begin{equation*}
A({\bm x}) = 10^{4} - 4.0 + 10^{4}\times \sin\big(\frac{\pi}{4}+\lfloor x_{1}+x_{2}\rfloor + \lfloor \frac{x_{1}}{\varepsilon}\rfloor + \lfloor \frac{x_{2}}{\varepsilon}\rfloor\big)
\end{equation*}
with $\varepsilon = 1/40$ as illustrated in \cref{fig5-1} (right), which exhibits a high-contrast feature and a multiscale structure. The source term $f({\bm x})$ is given by 
\begin{equation*}
f({\bm x}) = 10^{4}\times\exp\big(-10(x_{1}-0.5)^{2}-10(x_{2}-0.5)^{2}\big).
\end{equation*}
\vspace{-2ex}
}
\end{itemize}

The fine mesh is defined on a uniform Cartesian grid with $h=1/400$ for two examples on which all local computations are performed. The domain is first partitioned into $M=m^{2}$ square non-overlapping domains resolved by the mesh, and then overlapped by 2 layers of mesh elements to form an overlapping decomposition $\{ \omega_{i}\}$. Each overlapping subdomain $\omega_{i}$ is extended by $\ell$ layers of mesh elements to create a larger domain $\omega_{i}^{\ast}$ on which the local approximation space is built such that $\delta^{\ast}=2\ell h$. We use $s$ eigenfunctions of the Steklov eigenproblem \cref{eq:4-17} to build the discrete $A$-harmonic space on each oversampling domain $\omega_{i}^{\ast}$ and then construct the local approximation space for the GFEM by $n_{\rm loc}$ eigenfunctions of the eigenproblem \cref{eq:4-12}. Since no analytical solution of \cref{eq:5-1} is available, the standard finite element approximation $u_{h}$ on the fine mesh is considered as the reference solution. The error between the reference solution $u_{h}$ and the GFEM approximation $u^{G}_{h}$ is defined as 
\begin{equation}
\mathbf{error} := \frac{\Vert u_{h} - u_{h}^{G}\Vert_{a}}{\Vert u_{h}\Vert_{a}}.
\end{equation}

\begin{figure}
\begin{center}
\includegraphics[scale=0.32] {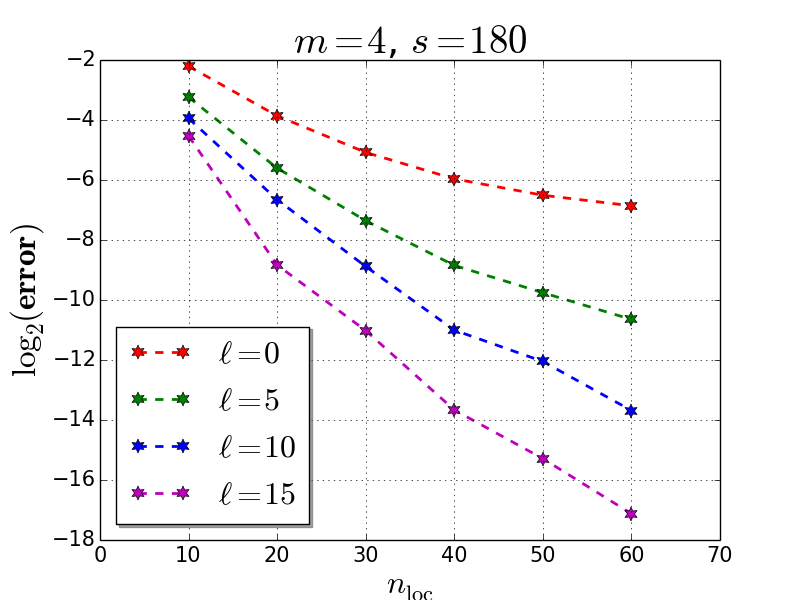}~\hspace{-5mm}
\includegraphics[scale=0.32] {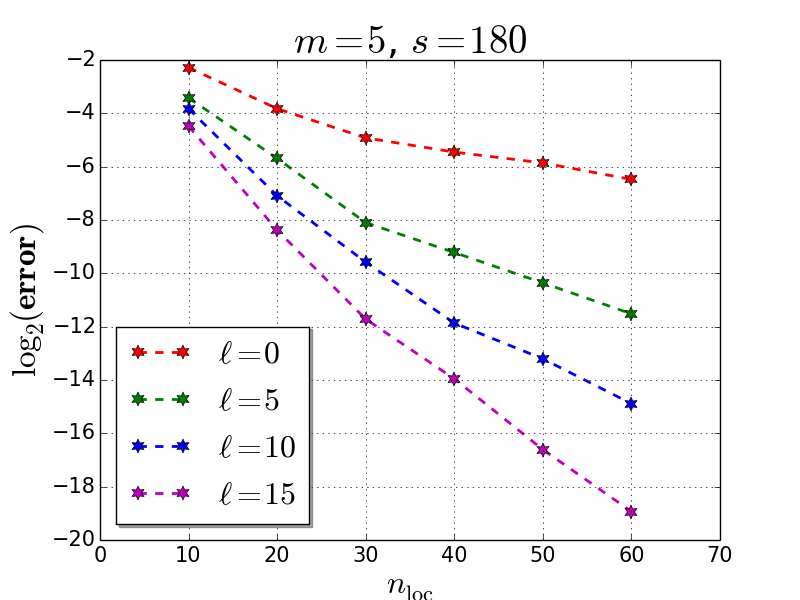}
\caption{Plots of $\log_{2}(\mathbf{error})$ against $n_{\rm loc}$ for the random field example (left) and the high contrast example (right). $n_{\rm loc}$ is the dimension of local spaces. }\label{fig5-2}
\end{center}
\end{figure}

\begin{figure}
\begin{center}
\includegraphics[scale=0.32] {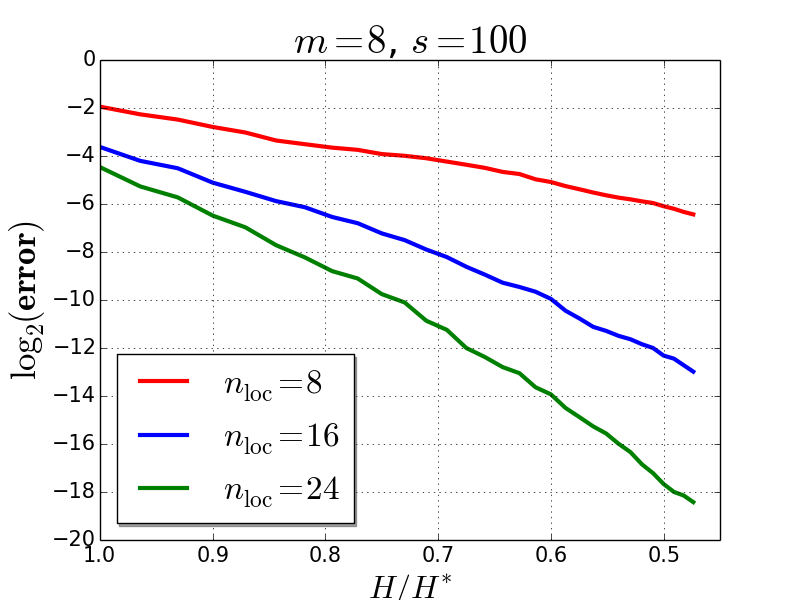}~\hspace{-5mm}
\includegraphics[scale=0.32] {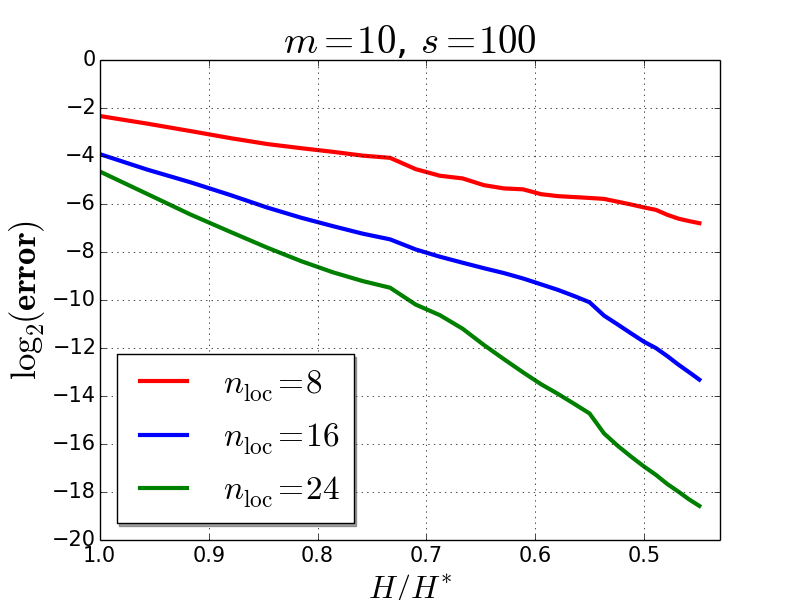}
\caption{Plots of $\log_{2}(\mathbf{error})$ against ${H}/{H}^{\ast}$ for the random field example (left) and the high contrast example (right). $H$ and ${H}^{\ast}$ represent the side lengths of the subdomains and the oversampling domains, respectively.}\label{fig5-3}
\end{center}
\end{figure}

%\begin{figure}
%\begin{center}
%\includegraphics[scale=0.4] {./new_fig/coarse_O.png}
%\caption{Plot of $\log_{2}(\mathbf{error})$ against $\mathcal{H}/\mathcal{H}^{\ast}$ for different values of $\mathcal{M}$. $\mathcal{H}$ and $\mathcal{H}^{\ast}$ represent the sizes of subdomains and oversampling domains, respectively. $\mathcal{N} = 8$, $\mathcal{D} = 100$.}\label{fig5-3}
%\end{center}
%\end{figure}

In \cref{fig5-2}, we plot the errors as functions of the dimension of local spaces for different oversampling sizes for the two examples on a semilogarithmic scale. We clearly see that the errors of both examples drop significantly with increasing oversampling sizes for a fixed $n_{\rm loc}$ and that the rate of convergence with respect to $n_{\rm loc}$ is higher with a larger oversampling size. This verifies our theoretical analysis. Moreover, we observe that even without oversampling ($\ell=0$), our method still converges. 

Next we test our method with different oversampling sizes $\ell$ to see how the error varies with $H/H^{\ast}$, where $H$ and $H^{\ast}$ defined by
\begin{equation}
H=\frac{\lceil \frac{400}{m}\rceil + 4}{400}\quad {\rm and}\quad H^{\ast}=\frac{\lceil \frac{400}{m}\rceil + 4 + 2\ell}{400}
\end{equation}
represent the side lengths of the subdomains $\{\omega_{i}\}$ and the oversampling domains $\{\omega^{\ast}_{i}\}$, respectively. In \cref{fig5-3}, the errors are plotted against $H/H^{\ast}$ for the two examples again on a semilogarithmic scale. We find that the rate of convergence of the error for a fixed $n_{\rm loc}$ is nearly exponential with respect to $H/H^{\ast}$ and the convergence rate is higher with larger $n_{\rm loc}$, which agrees well with our analysis; see \cref{eq:1-5-7}.
\begin{figure}
\centering
\includegraphics[scale=0.32] {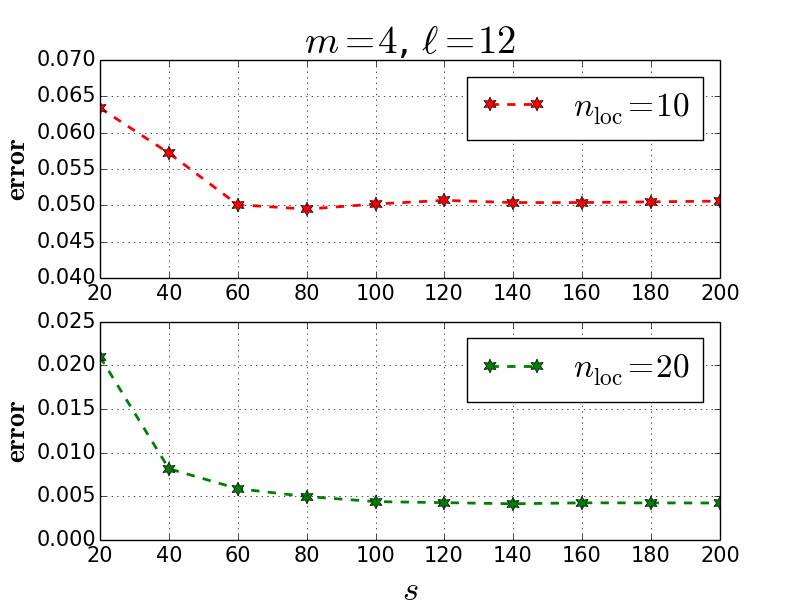}~\hspace{-5mm}
\includegraphics[scale=0.32]{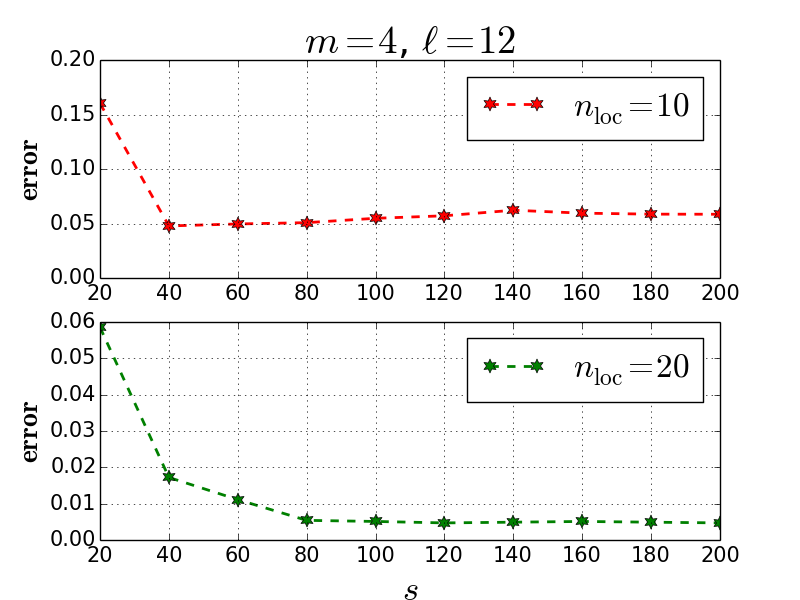}
\caption{Plots of $\mathbf{error}$ against the number of discrete $A$-harmonic basis functions used for the random field example (left) and the high contrast example (right).}\label{fig5-4}
\end{figure} 

To test our method of generating the discrete $A$-harmonic spaces via the Steklov eigenproblems, we let the number $n_{\rm loc}$ of eigenfunctions used for constructing each local space fixed and vary the number $s$ of discrete $A$-harmonic basis functions used. The errors are plotted in \cref{fig5-4} with $m=4$ and $\ell=12$ for the two examples. \cm{At first, the number of Steklov eigenfunctions used is very small and the error arising from the approximation of the discrete $A$-harmonic spaces dominates. When sufficiently many Steklov eigenfunctions are used, the error from approximating the discrete $A$-harmonic spaces is negligible, leading to the horizontal asymptotes in \cref{fig5-4}.} In this case, the true dimension of the discrete $A$-harmonic space $W_{h}(\omega_{i}^{\ast})$ is about 500. We see that a small number of discrete $A$-harmonic basis functions are capable to produce good numerical results for both examples. 

Finally, we plot the reference solutions $u_{h}$ and the errors $|u_{h}^{G} -u_{h}|$ (as a field) for the two examples in \cref{fig5-5}, where the multiscale approximate solutions $u_{h}^{G}$ are both computed with $n_{\rm loc}=20$, $\ell=10$, $m=4$, and $s = 80$. It can be observed that in this computational setting the multiscale approximate solutions agrees very well with the reference solutions for both examples. Furthermore, the error of the high contrast example is even smaller than that of the random field example, which demonstrates the effectiveness of our method for high-contrast problems.

%\begin{figure}[!htbp]
%\centering
%\includegraphics[scale=0.32]{./new_fig/p_sol.png}~
%\includegraphics[scale=0.32]{./new_fig/homo_sol.png}
%\caption{The global particular solution $u^{p}_{h}$ (left) and the homogeneous solution $u^{s}_{h}$ (right).}\label{fig5-6}
%\end{figure}

\begin{figure}[!htbp]
\centering
\includegraphics[scale=0.32]{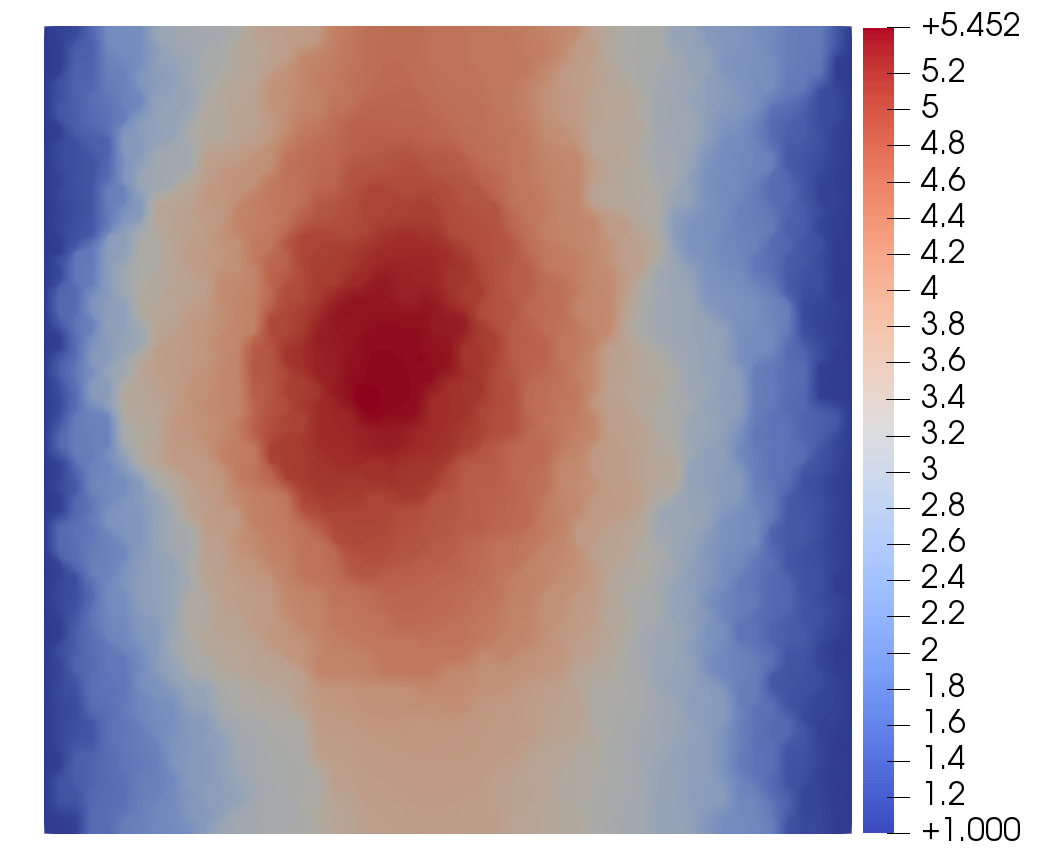}~
\includegraphics[scale=0.325]{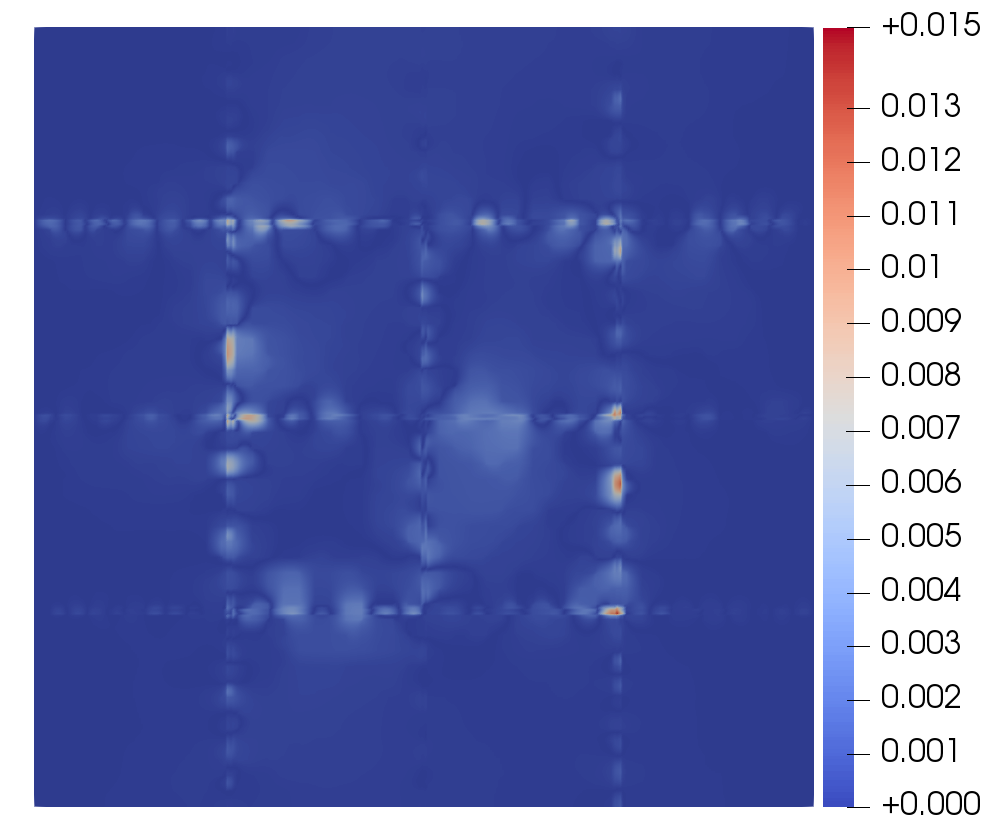}
\hspace{5mm}
\includegraphics[scale=0.32]{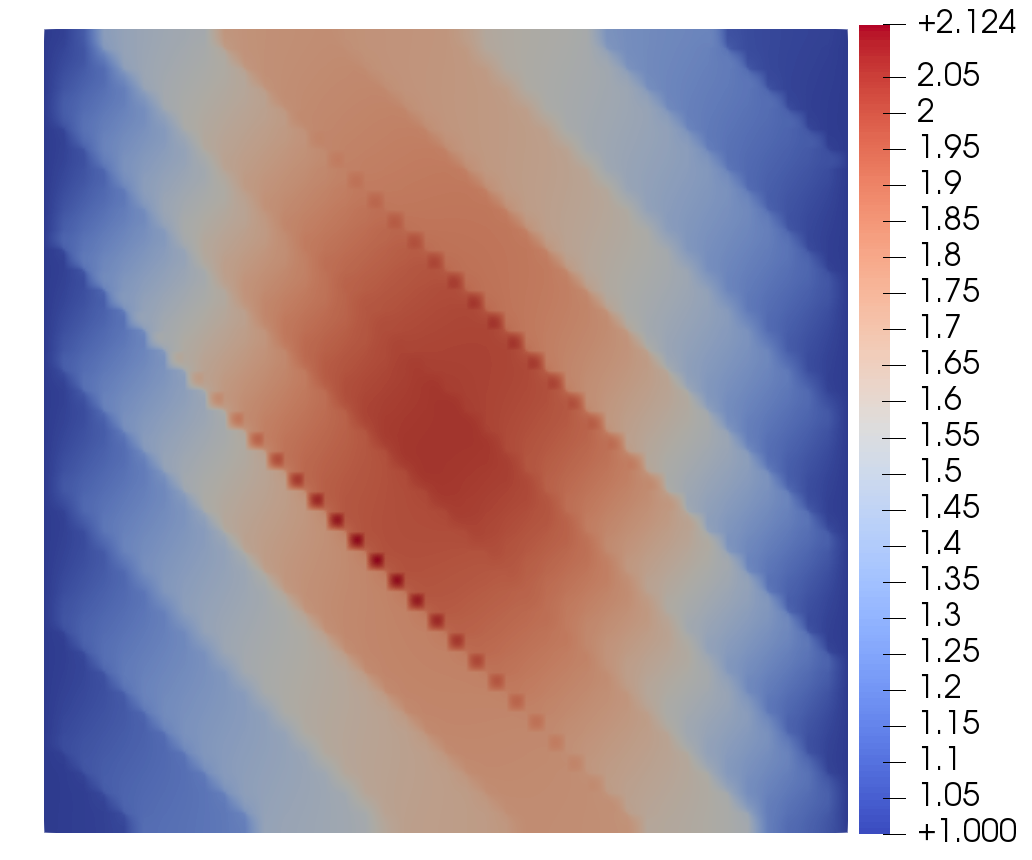}~
\includegraphics[scale=0.325]{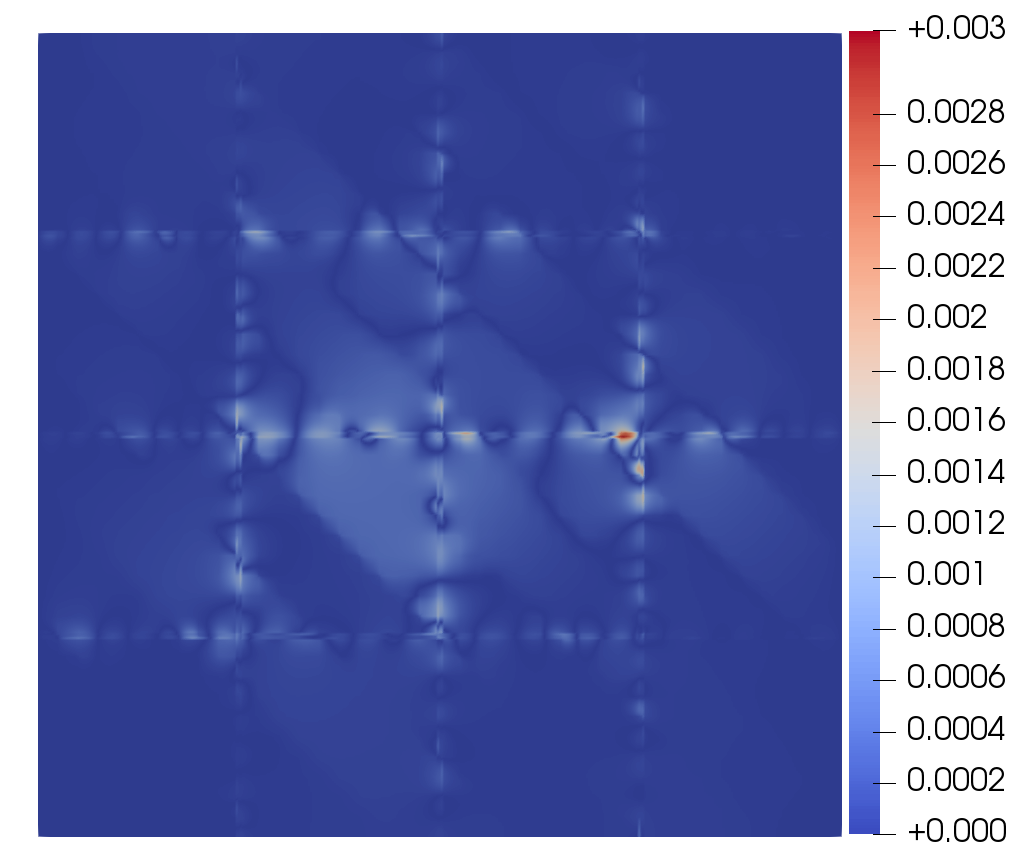}
\caption{The reference solution $u_{h}$ (left) and the error $|u_{h}-u_{h}^{G}|$ (as a field) (right) for the random field example (top) and the high contrast example (bottom).}\label{fig5-5}
\end{figure}

\section{Conclusions}\label{sec-6}
We have proposed new optimal local approximation spaces for the MS-GFEM based on local singular value decompositions of the partition of unity operators $Pu=\chi_{i}u$. An important feature of our method is that we approximate $\chi_{i}u$ instead of the exact solution $u$ locally within the GFEM scheme. Concerning theoretical aspects, we have given a rigorous proof of the nearly exponential decay rate for problems with mixed boundary conditions defined on general Lipschitz domains and investigated the influence of the oversampling size on the rate of convergence of the optimal local approximation, which had been missing in previous studies. Concerning practical aspects, we have proposed an easy-to-implement method for generating the local discrete $A$-harmonic spaces with a substantial reduction in computational cost by solving Steklov eigenproblems. It is important to note that the method and theory developed in this paper can be easily generalized to other positive definite PDEs in which a Caccioppoli-type inequality and Weyl asymptotics for the related eigenvalue problem are available (see the proof of \cref{lem:2-1}). Furthermore, in contrast to other multiscale methods (e.g., the LOD method \cite{maalqvist2014localization} and the MsFEM \cite{hou1997multiscale}), where the size of the coarse mesh is required to be small enough (at least theoretically) to attain convergence, the convergence of the MS-GFEM is guaranteed for an arbitrary coarse mesh provided that sufficiently many eigenfunctions are used for the local approximations.

Building on the work in this paper, in the near future we will investigate the discrete error estimate of the method, which has not been touched upon yet in this paper. Another focus of future work is the theoretical investigation of the contrast ($\beta/\alpha$) independent rate of convergence of the method. Numerical expriments in \cite{babuvska2020multiscale} have shown that the nearly exponential convergence rate of the MS-GFEM for composite materials is independent (or nearly independent) of the contrast. 

\appendix

\section{Proof of \cref{lem:2-1}}
\label{sec:A.3}

\begin{proof}
Define
\begin{equation}
\widetilde{H}_{A}(\omega^{\ast}) = \big\{v\in H_{A}(\omega^{\ast})\;: \;(v,1)_{L^{2}(\omega^{\ast})} = 0  \big\}.
\end{equation}
In view of the decomposition $H_{A}(\omega^{\ast})=\widetilde{H}_{A}(\omega^{\ast}) \oplus \mathbb{R}$, we can decompose $u\in H_{A}(\omega^{\ast})$ and $W_{m}(\omega^{\ast})\subset H_{A}(\omega^{\ast})$ in this way as $u = \tilde{u} + c$ and $W_{m}(\omega^{\ast}) = \widetilde{W}_{m}(\omega^{\ast}) \oplus \mathbb{R}$. It suffices to prove (\ref{eq:2-0-0}) with $u=\tilde{u}\in \widetilde{H}_{A}(\omega^{\ast})$ and $W_{m}(\omega^{\ast}) = \widetilde{W}_{m}(\omega^{\ast})\subset \widetilde{H}_{A}(\omega^{\ast})$.

We want to find an upper bound for the quantity
\begin{equation}\label{eq:2-3}
R= \sup_{u\in \widetilde{H}_{A}(\omega^{\ast})} \inf_{v\in \widetilde{W}_{m}(\omega^{\ast})} \frac{\Vert u -v\Vert_{L^{2}(\omega^{\ast})}}{\Vert u \Vert_{a,\omega^{\ast}}}.
\end{equation}
Fix $u\in \widetilde{H}_{A}(\omega^{\ast})$ and choose $v = \mathcal{P}_{m}u$, where $\mathcal{P}_{m}u$ denotes the projection of $u$ onto $\widetilde{W}_{m}(\omega^{\ast})$ with respect to the norm $\Vert \cdot \Vert_{a,\omega^{\ast}}$. Since $\Vert u - \mathcal{P}_{m}u\Vert_{a,\omega^{\ast}} \leq \Vert u \Vert_{a,\omega^{\ast}}$, it follows that
\begin{equation}\label{eq:2-4}
R \leq \sup_{u\in \widetilde{H}_{A}(\omega^{\ast})}  \frac{\Vert u -\mathcal{P}_{m}u\Vert_{L^{2}(\omega^{\ast})}}{\Vert u -  \mathcal{P}_{m}u \Vert_{a,\omega^{\ast}}} \leq \sup_{u\in \widetilde{W}^{\bot}_{m}(\omega^{\ast})}  \frac{\Vert u \Vert_{L^{2}(\omega^{\ast})}}{\Vert u \Vert_{a,\omega^{\ast}}},
\end{equation}
where $ \widetilde{W}^{\bot}_{m}(\omega^{\ast})$ denotes the orthogonal complement of $\widetilde{W}_{m}(\omega^{\ast})$ in $ \widetilde{H}_{A}(\omega^{\ast})$, i.e.,
\begin{equation}\label{eq:2-5}
\begin{array}{lll}
{\displaystyle  \widetilde{W}^{\bot}_{m}(\omega^{\ast}) = \{u\in  \widetilde{H}_{A}(\omega^{\ast})\, :\, a_{\omega^{\ast}}(u, v) = 0, \;\;\forall v\in  \widetilde{W}_{m}(\omega^{\ast})\} }\\[2mm]
{\displaystyle \qquad \qquad = \{u\in  \widetilde{H}_{A}(\omega^{\ast})\, :\, a_{\omega^{\ast}}(u, v) = 0, \;\;\forall v\in {W}_{m}(\omega^{\ast})\} }\\[2mm]
{\displaystyle\qquad \qquad  = \{u\in \widetilde{H}_{A}(\omega^{\ast})\, :\, a_{\omega^{\ast}}(u, \mathcal{P}^{A}v) = 0, \;\;\forall v\in \Psi_{m}(\omega^{\ast})\}. }
\end{array}
\end{equation}
By the definition of the projection $\mathcal{P}^{A}$, we have $a_{\omega^{\ast}}(u, \mathcal{P}^{A}v) = a_{\omega^{\ast}}(u, v)$ for all $u\in H_{A}(\omega^{\ast})$. Consequently,
\begin{equation}\label{eq:2-6}
\begin{array}{lll}
{\displaystyle \widetilde{W}^{\bot}_{m}(\omega^{\ast}) = \{u\in \widetilde{H}_{A}(\omega^{\ast})\,: \,a_{\omega^{\ast}}(u, v) = 0, \;\;\forall v\in \Psi_{m}(\omega^{\ast})\} }\\[2mm]
{\displaystyle \qquad \quad \quad =  \{u\in H_{A}(\omega^{\ast})\, :\, (u, v)_{L^{2}(\omega^{\ast})} = 0, \;\;\forall v\in \Psi_{m}(\omega^{\ast})\} }\\[2mm]
{\displaystyle  \subset  \Psi^{\bot}_{m}(\omega^{\ast}) =\{u\in H^{1}(\omega^{\ast})\, :\, (u, v)_{L^{2}(\omega^{\ast})} = 0, \;\;\forall v\in \Psi_{m}(\omega^{\ast})\} }.
\end{array}
\end{equation}
It follows that
\begin{equation}\label{eq:2-7}
\displaystyle R  \leq \sup_{u\in \Psi^{\bot}_{m}(\omega^{\ast})}  \frac{\Vert u \Vert_{L^{2}(\omega^{\ast})}}{\Vert u \Vert_{a,\omega^{\ast}}} = \Big(\inf_{u\in \Psi^{\bot}_{m}(\omega^{\ast})}  \frac{\Vert u \Vert_{a,\omega^{\ast}}}{\Vert u \Vert_{L^{2}(\omega^{\ast})}}\Big)^{-1}.
\end{equation}
By the minimum principle of the $(m+1)$-th eigenvalue, we see that $R\leq \lambda_{m+1}^{-1/2}$, where $\lambda_{m+1}$ is the largest eigenvalue associated with $\Psi_{m+1}(\omega^{\ast})$. Since the coefficient $A({\bm x})$ satisfies $A({\bm x})\xi \cdot\xi\geq \alpha|\xi|^{2}$ for all ${\bm x}\in \Omega$ and $\xi \in \mathbb{R}^{d}$, applying the comparison principle of eigenvalue problem to \cref{eq:2-1} and the Neumann Laplacian eigenproblem
\begin{equation}\label{eq:2-7-r}
(\nabla v, \nabla \varphi)_{L^{2}(\omega^{\ast})} = \mu(v,\,\varphi)_{L^{2}(\omega^{\ast})}, \quad \forall \varphi\in H^{1}(\omega^{\ast}),
\end{equation}
we get $\lambda_{m+1}  \geq \alpha\mu_{m+1}$, where $\mu_{m+1}$ is the $(m+1)$-th eigenvalue of \cref{eq:2-7-r}. By a classical asymptotic estimate $\mu_{m+1} = 4\pi \big(C(m)H^{\ast}\gamma_{d}^{1/d}\big)^{-2}$ \cite[Chapter VI]{Courant1989}, we obtain
\begin{equation}\label{eq:2-7-0}
R\leq C(m)H^{\ast}\frac{\gamma_{d}^{1/d}}{\sqrt{4\pi}}\alpha^{-1/2},
\end{equation}
where $H^{\ast}$ is the side length of the cube $\omega^{\ast}$, $\gamma_{d}$ is the volume of the unit ball in $\mathbb{R}^{d}$, and $C(m) = m^{-1/d}(1+o(1))$. It follows from \cref{eq:2-3,eq:2-7-0} that for any $u\in \widetilde{H}_{A}(\omega^{\ast})$, there exists a $v_{u}\in \widetilde{W}_{m}(\omega^{\ast})$ such that
\begin{equation}
\Vert u - v_{u}\Vert_{L^{2}(\omega^{\ast})} = \inf_{v\in \widetilde{W}_{m}(\omega^{\ast})} \Vert u - v\Vert_{L^{2}(\omega^{\ast})} \leq C(m)H^{\ast}\frac{\gamma_{d}^{1/d}}{\sqrt{4\pi}}\alpha^{-1/2}\Vert u \Vert_{a,\omega^{\ast}},
\end{equation}
which completes the proof. \qquad \end{proof}

\bibliographystyle{siamplain}
\bibliography{references}
\end{document}